\documentclass[english,10pt]{article}

\usepackage{amsmath,amsthm}
\usepackage{amssymb}
\usepackage{amsfonts}
\usepackage{amscd}
\newtheorem{thm}{Theorem}[section]
\newtheorem{theorem}[thm]{Theorem}

\newtheorem{corollary}[thm]{Corollary}
\newtheorem{lemma}[thm]{Lemma}
\newtheorem{proposition}[thm]{Proposition}

\theoremstyle{definition}
\newtheorem{definition}[thm]{Definition}

\newtheorem{examples}[thm]{Examples}
\newtheorem{remark}[thm]{Remark}
\newtheorem{remarks}[thm]{Remarks}

\newtheorem{notation}[thm]{Notation}
\newtheorem{free text}[thm]{}

\newcommand{\N} {\mathbb{N}}
\newcommand{\Z} {\mathbb{Z}}

\newcommand{\KK} {\mathbb{K}}
\newcommand{\bk} {\Bbbk}

\newcommand{\bt} {\mathbf{t}}
\newcommand{\bu} {\mathbf{u}}
\newcommand{\bG} {\mathbf{G}}

\newcommand{\bT} {\mathbf{T}}

\newcommand{\fg} {\mathfrak{g}}
\newcommand{\fh} {\mathfrak{h}}

\newcommand{\kzg} {K_{\Z}(\fg)}
\newcommand{\kzag} {K_{\Z_a}\!(\fg)}

\newcommand{\alg} {\mathrm{(alg)}}
\newcommand{\salg} {\mathrm{(salg)}}

\newcommand{\sets}{{(\mathrm{sets})}}
\newcommand{\grps} {\mathrm{(groups)}}
\newcommand{\spec}{{\hbox{\sl Spec}\,}}
\newcommand{\uspec}{\underline{\hbox{\sl Spec}\,}}

\newcommand{\Hom}{\mathrm{Hom}}

\newcommand{\lra} {\longrightarrow}

\newcommand{\End}{\mathrm{End}}

\newcommand{\rGL}{\mathrm{GL}}
\newcommand{\Lie}{\mathrm{Lie}}

\newcommand{\rsl}{\mathfrak{sl}}
\newcommand{\rgl}{\mathfrak{gl}}

\newcommand{\cL}{\mathcal{L}}
\newcommand{\cO}{\mathcal{O}}
\newcommand{\undt}{\underline{t}}

\font \smallrm=cmr10 at 11truept
 at 11truept
\font \smallsl=cmsl10 at 11truept

\begin{document}

{\ }

\vskip-89pt

\hfill
  {\smallrm {\smallsl Proceedings of the Edinburgh Mathematical Society\/}  (to appear)}
%
% {\smallbf 2}  (2008), 449--496}
%%
%
 \hskip19pt   {\ }
                                      \par
\hfill  {\ } 
%%
%   {\smallrm  {\smallbf DOI:}  10.1007/978-3-8348-9831-9\underline{\ }4}
%%
%

\hskip19pt   {\ }

\vskip41pt

\centerline{\Large \bf CHEVALLEY SUPERGROUPS}
 \vskip9pt
\centerline{\Large \bf OF TYPE  $ D(2,1;a) $}

\vskip29pt

\centerline{ F. Gavarini }

\vskip11pt

\centerline{\it Dipartimento di Matematica, Universit\`a di Roma ``Tor Vergata'' } \centerline{\it via della ricerca scientifica 1  --- I-00133 Roma, Italy}

\centerline{{\footnotesize e-mail: gavarini@mat.uniroma2.it}}

\vskip57pt

\begin{abstract}
 We present a construction ``\`a la Chevalley'' of connected affine supergroups associated with Lie superalgebras of type  $ D(2,1;a) $,  for any possible value of the parameter  $ a \, $.  This extends the results in  \cite {fg1}  (and  \cite{fg2}) where all other simple Lie superalgebras of classical type were considered.  The case of simple Lie superalgebras of Cartan type being dealt with in  \cite{ga},  so this work completes the program of constructing connected affine supergroups associated with any simple Lie superalgebra.
\end{abstract}

\vskip45pt

\section{Introduction}

\smallskip

   {\ }
 \footnote{\ 2010 {\it MSC}\;: \, Primary 14M30, 14A22; Secondary 17B20.}
 In his work of 1955, Chevalley provided a combinatorial construction of all simple affine algebraic groups over any field.  In particular, his method led to an existence theorem for simple affine algebraic groups: one starts with a simple (complex, finite-dimensional) Lie algebra and a simple module  $ V $  for it, and realizes the required group as a closed subgroup of  $ \rGL(V) \, $.  This can also be recast as to provide a description of all simple affine groups as group schemes over  $ \Z \, $.
                                                                \par
   In  \cite{fg1}  the philosophy of Chevalley was revisited in the context of supergeometry.  The outcome is a construction of affine supergroups whose tangent Lie superalgebra is of  {\sl classical type}.  However, some exceptions were left out, namely the cases when the Lie superalgebra is of type  $ D(2,1;a) $  and the parameter  $ a $  is not an integer number; the present work fills in this gap.  As the case of simple Lie superalgebras of Cartan type is solved in  \cite{ga},  this paper completes the program of constructing connected affine supergroups associated with any simple Lie superalgebra.

\smallskip

   By ``affine supergroup'' here I mean a representable functor from the category  $ \salg $  of commutative superalgebras   --- over some fixed ground ring ---   to the category  $ \grps $  of groups: in other words, an affine supergroup-scheme, identified with its functor of points.
%
% Thus we adopt the functorial language (\`a la Grothendieck, say) from scratch.  The general procedure
% developed in  \cite{fg1}  proceeds in two steps: first, one
%
 In  \cite{fg1},  one first constructs a functor from  $ \salg $  to  $ \grps \, $,  recovering Chevalley's ideas to define the values of such a group functor on each superalgebra  $ A $   --- i.e., to define its  $ A $--points;  then one proves that the sheafification of this functor is representable   --- hence it is an affine supergroup-scheme.

\smallskip

   For the case  $ D(2,1;a) $   --- with  $ \, a \not\in \Z \, $  ---   one needs a careful modification of the general procedure of  \cite {fg1};  thus the presentation hereafter will detail those steps which need changes, and will simply refer to  \cite{fg1}  for those where the original arguments still work unchanged.

\medskip

   The initial datum is a simple Lie superalgebra  $ \, \fg = D(2,1;a) \, $.
                                                          \par
   We start with basic results on  $ \fg \, $:  the existence of  {\sl Chevalley bases\/}  (with nice integrality properties) and a PBW theorem for the Kostant $ \Z $--form  of the universal enveloping superalgebra  $ U(\fg) \, $.
                                                          \par
   Next we take a faithful, finite-dimensional  $ \fg $--module $ V \, $,  and we show it has suitable lattices  $ M $  invariant by the Kostant superalgebra.  This allows to define   --- functorially ---   additive and multiplicative one-parameter (super)subgroups of operators acting on scalar extensions of  $ M $.  The additive subgroups are just like in the general case: there exists one of them for every root of  $ \fg \, $.  The multiplicative ones instead are associated to elements of the fixed Cartan subalgebra of  $ \fg \, $,  and are of two types: those of  {\sl classical\/}  type, modeled on the group functor  $ \, A \mapsto U(A_0) \, $   --- the  {\sl group of units\/}  of  $ A_0 $  ---   and those of  {\sl  $ a $--type},  modeled on the group functor  $ \, A \mapsto P_a(A) \, $   --- the group of elements of  $ A_0 $  ``which may be raised to the $ a^k $--th  power, for  all  $ k \, $''.  The second type of multiplicative one-parameter subgroups, not used in  \cite{fg1},  is now needed because one has to consider the ``operation''  $ \, t \mapsto t^a \, $,
%
% \, which in general is not defined for any  $ \, t \in U(A_0) \, $,  but it is
%
 defined just
for  $ \, t \in P_a(A) \, $;  this marks a difference with the case  $ \, a \in \Z \, $.
                                                          \par
   Then we consider the functor  $ \, G : \salg \longrightarrow \grps \, $  whose value  $ G(A) $  on  $ \, A \in \salg \, $  is the subgroup of  $ \rGL(V(A)) $   --- with  $ \, V(A) := A \otimes M \, $  ---   generated by all the homogeneous one-parameter supersubgroups mentioned above.  This functor is a presheaf, hence we can take its sheafification  $ \, \bG_V \! = \bG : \salg \longrightarrow \grps \, $.  These  $ \bG_V $  are, by definition, our ``Chevalley supergroups''.
                                                          \par
   Acting just like in  \cite{fg1},  one defines a ``classical affine subgroup''  $ \bG_0 $  of  $ \bG_V \, $,  corresponding to the even part  $ \fg_0 $  of  $ \fg $  (and to  $ V $),  and then finds a factorization  $ \; \bG_V = \bG_0 \, \bG_1 \cong \bG_0 \times \bG_1 \, $,  where  $ \bG_1 $  corresponds instead to the odd part $ \fg_1 $  of  $ \fg \, $.  Actually, one has even a finer factorization  $ \, \bG_V = \bG_0 \times \bG_1^{-,<} \times \bG_1^{+,<} \, $  with  $ \bG_1^{\pm,<} $  being totally odd superspaces associated to the positive or negative odd roots of  $ \fg \, $.  Thus  $ \, \bG_1 = \bG_1^{-,<} \times \bG_1^{+,<} \, $  is representable, and  $ \bG_0 $  is representable too, hence the above factorization implies that  $ \bG_V $  is representable too, so it is an affine supergroup.  The outcome then is that our Chevalley supergroups are affine supergroups.

\smallskip

   Finally, one also proves that our construction is functorial in  $ V $  and that  $ \Lie(\bG_V) $  is just  $ \fg $  as one expects, like in  \cite{fg1}  (no special changes are needed).

\vskip21pt

   \centerline{\bf Acknowledgements}
 \vskip5pt
   The author thanks M.~Duflo, R.~Fioresi and C.~Gasbarri for many useful conversa\-tions, and K.~Iohara and Y.~Koga for their valuable explanations.

\bigskip

\section{Preliminaries}  \label{preliminaries}

\smallskip

%
%    {\ } \quad   In this section we present some basic preliminaries of supergeometry.
% Further details are in  \cite{ccf}.
%

 \subsection{Superalgebras, superspaces, supergroups}  \label{first_preliminaries}

\smallskip

   {\ } \quad   Let  $ \bk $  be a unital, commutative ring.

\vskip5pt

   We call  {\it  $ \bk $--superalgebra\/}  any associative, unital  $ \bk $--algebra  $ A $  which is  $ \Z_2 $--graded.  Thus  $ A $  splits as  $ \; A = A_0 \oplus A_1 \; $,  and  $ \; A_a \, A_b \subseteq A_{a+b} \; $.  The  $ \bk $--submodule  $ A_0 $  and its elements are called  {\it even},  while  $ A_1 $  and its elements  {\it odd}\,.  By  $ p(x) $  we denote the  {\sl parity\/}  of any homogeneous element  $ \, x \in A_{p(x)} \, $.  A superalgebra  $ A $  is said to be  {\it commutative\/}  iff  $ \; x y = (-1)^{p(x)p(y)} y x \; $  for all homogeneous  $ \, x $,  $ y \in A \, $,  and  $ \, z^2 = 0 \, $  for all odd  $ \, z \in A_1 \, $.
%                                                    \par
   Clearly,  $ \bk $--superalgebras  form a category, whose morphisms are those in the category of algebras which preserve the unit and the  $ \Z_2 $--grading; we denote  $ \salg $  the full subcategory of commutative superalgebras; we shall also write  $ \salg_\bk $  to stress the role of  $ \bk \, $.
                                                   \par
   At last, for any  $ \, n \in \N \, $  we call  $ A_1^{\,n} $  the  $ A_0 $--submodule  of  $ A $  spanned by all products  $ \, \vartheta_1 \cdots \vartheta_n \, $  with  $ \, \vartheta_i \in A_1 \, $  for all  $ i \, $,  and  $ A_1^{(n)} $  the unital  $ \bk $--subalgebra  of  $ A $  generated by  $ A_1^{\,n} \, $.

\vskip5pt

   We consider the notions of  {\sl superspace},  {\sl (affine) superscheme\/}  and  {\sl (affine) supergroup\/}  as defined in detail in  \cite{ccf}  and  \cite{fg1}:  here we just recall them rather quickly.  Roughly speaking, a  {\sl superspace\/}  is a locally ringed space whose structure sheaf is made of commutative superalgebras, the stalks being local; the morphisms among superspaces then are morphisms of locally ringed spaces, respecting the parity on sections of the structure sheaf.  In any superspace  $ S $,  the structure sheaf  $ \cO_S $  has natural ``even'' and ``odd'' parts, say  $ \cO_{S,0} $  and  $ \cO_{S,1} \, $:  the latter is a sheaf of modules over the former, which in turn is just a sheaf of commutative algebras (whose stalks are local).
                                                   \par
   A  {\sl superscheme\/}  is just a superspace  $ S $  for which  $ \cO_{S,1} $  is quasi-coherent as a  $ \cO_{S,0} $--module.

\vskip5pt

   For any  $ \, A \in \salg \, $,  the prime spectrum  $ \spec(A_0) $  of its even part bears a natural structure of superscheme, denoted by  $ \uspec(A) \, $,  induced by the fact that  $ A $  itself is an  $ A_0 $--module.  Any superscheme which is isomorphic to such a  $ \uspec(A) $  is said to be  {\sl affine}.
%
% \medskip
%
%    Clearly any superscheme is locally isomorphic to an affine superscheme.
%

\vskip9pt

   If  $ X $  is any superscheme, its  {\it functor of points\/}  is the functor  $ \; h_X : \salg \lra \sets \; $  defined on objects by  $ \; h_X(A) := \Hom \big(\, \uspec(A) \, , X \big) \; $  and on arrows by  $ \; h_X(f)(\phi) := \phi \circ \uspec (f) \; $.  If  $ \grps $  is the category of groups and  $ h_X $  is a functor  $ \; h_X : \salg \lra \grps \, $,  then we say that  $ X $  is a  {\it supergroup}.  When  $ X $  is affine, this is equivalent to the
fact that $ \, \cO(X) \, $   --- the superalgebra of global sections of the structure sheaf on  $ X $  ---   is a
(commutative)  {\sl Hopf superalgebra}. More in general, by  {\it supergroup functor\/}
we mean any functor  $ \; G : \salg \lra \grps \; $.

\vskip3pt

   Any representable supergroup functor is the same as an affine supergroup: indeed, the former corresponds to the functor of points of the latter.
%
% Indeed, the functor of points of a given superscheme grasps all the information carried by a superscheme
% (by Yoneda's Lemma): in particular, two superschemes are isomorphic if and only if their functors of points
% are.
%
 Thus we use the same letter to denote both a superscheme and its functor of points.  See  \cite{ccf},  Ch.~3--5, for more details.
%
% \smallskip
%
%    In the present work, we shall actually consider only affine supergroups, which we are going to describe
% via their functor of points.  The next examples turn out to be very important:
%

\medskip

\begin{examples} {\ }
 \vskip2pt
   {\it (1)} \,  The  {\sl affine superspace\/}  $ \, \mathbb{A}_\bk^{p|q} \, $,  also denoted  $ \, \bk^{p|q} \, $,  is defined  --- for each  $ \, p \, $, $ q \in \N \, $  ---   as  $ \; \mathbb{A}_\bk^{p|q} := \uspec \big( \bk[x_1,\dots,x_p] \otimes_\bk \bk[\xi_1 \dots \xi_q] \big) \; $;  hereafter  $ \bk[\xi_1 \dots \xi_q] $  is the exterior (or ``Grassmann'') algebra generated by  $ \xi_1 $,  $ \dots $,  $ \xi_q \, $,  and  $ \, \bk[x_1,\dots,x_p] \, $  the polynomial algebra in  $ p $  indeterminates.
 \vskip2pt
   {\it (2)} \,  Let  $ V $  be a free  $ \bk $--supermodule.  Set  $ \; V(A) \, := \, {(A \otimes V)}_0 \, = \, A_0 \otimes V_0 \oplus A_1 \otimes V_1 \; $  for any  $ \, A \in \salg \, $:  this yields a representable functor in the category of superalgebras, represented by the superalgebra of polynomial functions on  $ V \, $.  Hence  $ V $  can be seen as an affine superscheme.
 \vskip2pt
   {\it (3)} \,  {\sl  $ \rGL(V) $  as an affine supergroup}.  Let  $ V $  be a free  $ \bk $--supermodule  of finite (super)rank  $ p|q \, $.  For any  $ \, A \in \salg_\bk \, $,  let  $ \, \rGL(V)(A) := \rGL\big(V(A)\big) \, $  be the set of isomorphisms  $ \; V(A) \lra V(A) \; $.  If we fix a homogeneous basis for  $ V $,  we see that  $ \, V \cong \bk^{p|q} \; $;  in other words,  $ \, V_0 \cong \bk^p \, $  and  $ \, V_1 \cong \bk^q \, $.  In this case, we also denote $ \, \rGL(V) \, $  with  $ \, \rGL_{p|q} \, $.  Now,  $ \rGL_{p|q}(A) $  is the group of invertible matrices of size  $ (p+q) $  with diagonal block entries in  $ A_0 $  and off-diagonal block entries in $ A_1 \, $.  It is known that the functor  $ \rGL(V) $  is representable (see  \cite{vsv}, Ch.~3,  for details).   \hfill  $ \diamondsuit $
\end{examples}

\medskip

 \subsection{Lie superalgebras}  \label{Lie-superalgebras}

\smallskip

   {\ } \quad   The notion of Lie superalgebra is well known, at least  over a field of characteristic neither 2 nor 3.  To take into account  {\sl any ground ring},  we consider a modified formulation: it is a ``correct'' notion of Lie superalgebra given by the standard notion enriched with an additional  ``$ 2 $--mapping'', an analogue to the  $ p $--mapping  in a  $ p $--restricted  Lie algebra over a field of characteristic  $ p > 0 \, $.

\smallskip

\begin{definition}  \label{def-Lie-salg}
 {\it (cf.~\cite{bmpz}, \cite{du})} \,
  Let  $ \, A \in \salg_\bk \, $.  We call  {\sl Lie\/  $ A $--superalgebra\/}  any  $ A $--supermodule  $ \, \fg = \fg_0 \oplus \fg_1 \, $  endowed with a  {\it (Lie super)bracket\/}  $ \; [\,\ ,\ ] : \fg \times \fg \longrightarrow \fg \, $,  $ \; (x,y) \mapsto [x,y] \, $,  \, and a  {\it  $ 2 $--operation\/}  $ \; {(\ )}^{\langle 2 \rangle} : \fg_1 \longrightarrow \fg_0 \, $,  $ \; z \mapsto z^{\langle 2 \rangle} \, $,  \, such that (for all  $ \, x , y \in \fg_0 \cup \fg_1 \, $,  $ \, w \in \fg_0 \, $,  $ \, z, z_1, z_2 \in \fg_1 $):
 \vskip5pt
   {\it (a)}  \quad  $ [\,\ ,\ ] \, $  is  $ A $--superbilinear (in the obvious sense)\;,
\qquad  $ [w,w] \; = \; 0  \;\; ,  \qquad  \big[z[z,z]\big] \; = \; 0  \quad $;
 \vskip5pt
   {\it (b)}  $ \qquad  [x,y] \, + \, {(-1)}^{p(x) \, p(y)}[y,x] \; = \; 0  \qquad\; $  {\sl (anti-symmetry)}\,;
 \vskip7pt
   {\it (c)}  $ \quad  {(-\!1)}^{p(x) p(z)} [x,[y,z]] + {(-\!1)}^{p(y) p(x)} [y,[z,x]] \, + \, {(-\!1)}^{p(z) p(y)} [z,[x,y]] \, = \, 0 \;\;\, $  {\sl (Jacobi identity)}\,;
 \vskip6pt
   {\it (d)}  \;\quad  $ {(\ )}^{\langle 2 \rangle} \, $  is  $ A $--quadratic, i.e.~$ \;\;  {(a_0 \, z)}^{\langle 2 \rangle} = \, a^2 \, z^{\langle 2 \rangle} \;\; $,  $ \;\;  {(a_1 \, w)}^{\langle 2 \rangle} = \, 0 \;\;\; $  for  $ \; a_0 \in A_0 \, $,  $ \; a_1 \in A_1 \; $;
 \vskip6pt
%
%    {\it (e)}  \qquad  $ {(z_1 \! + z_2)}^{\langle 2 \rangle}  \, = \;
% z_1^{\langle 2 \rangle} + \, [z_1,z_2] \, + \, z_2^{\langle 2 \rangle} \quad $;
% %
%  \vskip6pt
% %
%    {\it (f)}  \qquad  $ \big[ z^{\langle 2 \rangle}, x \big]  \, = \;
% \big[ z \, , [z,x] \big] \quad $.
%
   {\it (e)}  \qquad  $ {(z_1 \! + z_2)}^{\langle 2 \rangle}  \, = \;  z_1^{\langle 2 \rangle} + \, [z_1,z_2] \, + \, z_2^{\langle 2 \rangle}  \quad ,  \quad \qquad  \big[ z^{\langle 2 \rangle}, x \big]  \, = \;  \big[ z \, , [z,x] \big] \quad $.
 \vskip9pt
   All Lie  $ A $--superalgebras  form a category, whose morphisms are the  $ A $--superlinear  (in the obvious sense), graded maps preserving the bracket and the  $ 2 $--operation.
 \vskip3pt
   A Lie superalgebra  $ \fg $  is called  {\it classical\/}  if it is simple, i.e.~it has no nontrivial (homogeneous) ideals, and  $ \fg_1 $  is semisimple as a  $ \fg_0 $--module.
%
% Furthermore,  $ \fg $  is said to be  {\it basic\/}  if, in addition, it admits
% a non-degenerate, invariant bilinear form.
%
 Classical Lie superalgebras of finite dimension over algebraically closed fields of characteristic zero were classified by V.~Kac  (cf.~\cite{ka}, \cite{sc}),  whom we shall refer to for the standard terminology and notions.
\end{definition}

\smallskip

\begin{examples}   \label{def-Lie/Ass_End(V)}
 {\it (a)} \,  Let  $ \, \mathcal{A} = \mathcal{A}_0 \oplus \mathcal{A}_1 \, $  be any associative  $ \bk $--superalgebra.  There is a canonical structure of Lie superalgebra on  $ \mathcal{A} $  given by  $ \,\; [x,y] \, := \, x\,y - {(-1)}^{p(x) p(y)} y\,x \;\, $  for all homogeneous  $ \, x, y \in \mathcal{A}_0 \cup \mathcal{A}_1 \; $
and  $ 2 $--operation  $ \,\; z^{\langle 2 \rangle} := z^2 = z\,z \;\, $  (the associative square in  $ \mathcal{A} $)  for all odd  $ \, z \in \mathcal{A}_1 \, $.
 \vskip4pt
   {\it (b)} \,  Let  $ \, V = V_0 \oplus V_1 \, $  be a  {\sl free\/}  $ \bk $--supermodule,  and consider  $ \End(V) \, $,  the endomorphisms of  $ V $  as an ordinary  $ \bk $--module.  This is again a free super  $ \bk $--module,  $ \; \End(V) = \End(V)_0 \oplus \End(V)_1 \, $,  \; where $ \End(V)_0 $  are the morphisms which preserve the parity, while  $ \End(V)_1 $  are the morphisms which reverse the parity.  By the recipe in  {\it (a)},  $ \End(V) $  is a Lie  $ \bk $--superalgebra
with  $ \; [A,B] := A B - {(-1)}^{p(A) p(B)} B A \, $,  $ \; C^{\langle 2 \rangle} := C^2 \, $,  \; for all  $ \, A, B, C \in \End(V) \; $  homogeneous, with  $ C $  odd.
                                                                         \par
   The standard example is for  $ V $  of finite rank, say  $ \, V := \bk^{p|q} = \bk^p \oplus \bk^q \, $,  with  $ \, V_0 := \bk^p \, $  and  $ \, V_1 := \bk^q \, $:  in this case we also write  $ \; \End\big(\bk^{m|n}\big) \! := \End(V) \; $  or  $ \; \rgl_{p|q} := \End(V) \; $.  Choosing a basis for  $ V $  of homogeneous elements (writing first the even ones), we identify  $ \End(V)_0 $  with the set of all diagonal block matrices, and  $ \End(V)_1 $  with the set of all off-diagonal block matrices.   \hfill  $ \diamondsuit $
\end{examples}

\medskip

 \subsection{The Lie superalgebra  $ D(2,1;a) $}  \label{def_D(2,1;a)}

\smallskip

   {\ } \quad   Let  $ \KK $  be an algebraically closed field of characteristic zero, and let  $ \, a \in \KK \setminus \{0,-1\} \, $.  Then let  $ \, \Z[a] \, $  be the unital subring of  $ \KK $  generated by  $ a \, $.  Clearly  $ \, \Z[a] = \Z \, $  if and only if  $ \, a \in \Z \, $.

\smallskip

   According to Kac's work, we can realize  $ \, \fg := D(2,1;a) \, $  as a contragredient Lie superalgebra: in particular, it admits a presentation by generators and relations with a standard procedure (detailed in general in  \cite{ka}).  In order to do that, we first fix a specific choice of Dynkin diagram and corresponding Cartan matrix, like in  \cite{fss}, \S 2.28  (first choice), namely
 \vskip-5pt
  $$  {\buildrel 2 \over {\textstyle \bigcirc}} \hskip-3,5pt
\joinrel\relbar\joinrel\relbar\hskip-7pt{\buildrel 1 \over -}\hskip-7pt\relbar\joinrel\relbar\joinrel
\hskip-3,5pt
{\buildrel 1 \over {\textstyle \bigotimes}}
\hskip-3,5pt
\joinrel\relbar\joinrel\relbar\hskip-7pt{\buildrel a \over -}\hskip-7pt\relbar\joinrel\relbar\joinrel
\hskip-3,5pt
{\buildrel 3 \over {\textstyle \bigcirc}}  \hskip17pt ,   \hskip35pt
  {\big( a_{i,j} \big)}_{i,j=1,2,3;}  \, :=  \begin{pmatrix}
                                              0  &  1  &  a \,  \\
                                             -1  &  2  &  0 \,  \\
                                             -1  &  0  &  2 \,
                                             \end{pmatrix}  $$
   We define  $ \, \fg = D(2,1;a) \, $  as the Lie superalgebra over  $ \KK $  with generators  $ \; h_i \, $,  $ \, e_i \, $,  $ \, f_i \;\; (i=1,2,3) \, $,  \, with degrees  $ \; p(h_i) := 0 \, $,  $ \, p(e_i) := \delta_{1,i} \, $,  $ \, p(f_i) := \delta_{1,i} \;\; (i=1,2,3) \, $,  and relations (for  $ \, i,j=1,2,3 \, $)
  $$  \displaylines{
   \big[ h_i, h_j \big] = 0 \;\; ,  \qquad  \big[ e_1, e_1 \big] = 0 \;\; ,  \qquad  \big[ f_1, f_1 \big] = 0 \;\; ,  \cr
   \big[ h_i, e_j \big] = +a_{i,j} \, e_j \;\; ,  \qquad  \big[ h_i, f_j \big] = -a_{i,j} \, f_j \;\; ,  \qquad  \big[ e_i, f_j \big] = \delta_{i,j} \, h_i \;\; ,  \cr
   e_1^{\langle 2 \rangle} = \, 0 \;\;\; ,  \qquad \qquad  f_1^{\,\langle 2 \rangle} = \, 0 \;\;\; .  }  $$

\vskip5pt

   The subspace  $ \, \fh := \sum_{i=1}^3 \KK \, h_i \, $  is a  {\sl Cartan subalgebra\/}  of  $ \fg $  (included in  $ \fg_0 $).  The adjoint action of  $ \fh $  splits  $ \fg $  into eigenspaces, namely
 $ \; \fg  \, = \! {\textstyle \bigoplus\limits_{\alpha \in \fh^*}} \fg_\alpha \; $  with  $ \, \fg_\alpha := \big\{ x \! \in \! \fg \;\big|\, [h,x] = \alpha(h) \, x \, , \; \forall \; h \! \in \! \fh \big\} \, $  for all  $ \, \alpha \in \fh^* \, $.
 Then we define the  {\sl roots\/}  of  $ \fg $  by
 $ \; \Delta  \, := \,  \Delta_0 \coprod \Delta_1  \, = \,  \{\,\text{{\sl roots\/}  of  $ \fg $}\,\} \; $  with
  $$  \displaylines{
   \Delta_0  \, := \,  \big\{\, \alpha \in \fh^* \setminus \{0\} \;\big|\; \fg_\alpha \cap \fg_0 \not= \{0\} \big\}  \, = \,  \{\,\text{{\sl even roots\/}  of  $ \fg $} \,\}  \cr
   \Delta_1  \, := \,  \big\{\, \alpha \in \fh^* \;\big|\; \fg_\alpha \cap \fg_1 \not= \{0\} \big\}  \, = \,  \{\,\text{{\sl odd roots\/}  of  $ \fg $} \,\}
%
% \cr
%    \Delta  \, := \,  \Delta_0 \coprod \Delta_1  \, = \,  \{\,\text{{\sl roots\/}  of  $ \fg $} \,\}
%
  }  $$
 Now  $ \Delta $  is called the  {\sl root system\/}  of  $ \fg \, $,  and for each root  $ \alpha $  we call  $ \fg_\alpha $  its  {\sl root space}.
%
% Indeed,  $ \Delta_0 $  is the root system of the (reductive) Lie algebra  $ \fg_0 \, $,  and
% $ \Delta_1 $  is the set of weights of the representation of  $ \fg_0 $  in  $ \fg_1 \, $.
%
 Moreover, every non-zero vector in a root space is called  {\sl root vector}.  Note that every root space is one dimensional, so any root vector forms a basis of its own root space.  Then any  $ \KK $--basis  of  $ \fh $  together with any choice of a root vector for each root will provide a  $ \KK $--basis  of  $ \, \fg = D(2,1;a) \, $.

\smallskip

   There is an even, non-degenerate, invariant bilinear form on  $ \fg \, $,  whose restriction to  $ \fh $  is in turn an invariant bilinear form on  $ \fh \, $.  We denote this form by $ \big( x, y \big) \, $,  and we use it to identify  $ \fh^* $  to  $ \fh $,  via  $ \, \alpha \mapsto H_\alpha \, $,  \, and then to define a similar form on  $ \fh^* $,  such that  $ \; \big( \alpha' , \alpha'' \big) = \big( H_{\alpha'}, H_{\alpha''} \big) \; $.  In particular, we fix normalizations so that  $ \, \alpha(H_\alpha) = 2 \, $  and  $ \, \alpha(H_{\alpha'}) = 0 \, $  for all  $ \, \alpha, \alpha' \in \Delta_0 \, $  such that  $ \, \big( \alpha \, , \alpha' \big) = 0 \, $  (in short, we adopt the normalizations as in  \cite{hu}).  Moreover, if  $ \alpha $  is any (even, odd, etc.) root we shall call the vector  $ H_\alpha $  the (even, odd, etc.)  {\sl coroot\/}  associated to  $ \alpha \, $.

\medskip

   Actually, we can describe explicitly the root system of  $ \, \fg = D(2,1;a) \, $  (after  \cite{fss}, Table 3.60) as
 $ \; \Delta = \{\, \pm 2 \, \varepsilon_1 \, , \, \pm 2 \, \varepsilon_2 \, , \, \pm 2 \, \varepsilon_3 \, , \, \pm \varepsilon_1 \pm \varepsilon_2 \pm \varepsilon_3 \,\} \, $,
 $ \; \Delta_0 = \{\, \pm 2 \, \varepsilon_1 \, , \, \pm 2 \, \varepsilon_2 \, , \, \pm 2 \, \varepsilon_3 \,\} \, $,
 $ \; \Delta_1 = \{\, \pm \varepsilon_1 \pm \varepsilon_2 \pm \varepsilon_3 \,\} \; $
where  $ \{\varepsilon_1,\varepsilon_2,\varepsilon_3\} $  is an orthogonal basis in a  $ \KK $--vector  space with inner product  $ \, (\ \,,\ ) \, $  such that  $ \; (\varepsilon_1,\varepsilon_1) = -(1+a)\big/2  \, $,  $ \; (\varepsilon_2,\varepsilon_2) = 1\big/2 \; $,  $ \; (\varepsilon_3,\varepsilon_3) = a\big/2 \, $.
 \vskip3pt
   Then we fix a distinguished system of  {\sl simple roots},  say  $ \, \{ \alpha_1 , \alpha_2 , \alpha_3 \} \, $,  namely
 $ \; \alpha_1 = \, \varepsilon_1 - \varepsilon_2 - \varepsilon_3 \; $,  $ \; \alpha_2 = \, 2 \, \varepsilon_2 \; $,  $ \; \alpha_3 = \, 2 \, \varepsilon_3 \; $
associated to this choice.  In terms of these, we call  {\sl positive\/}  the roots
  $$  \displaylines{
   2 \, \varepsilon_1 \, = \, 2 \, \alpha_1 + \alpha_2 + \alpha_3 \;\; ,  \qquad \qquad  2 \, \varepsilon_2 \, = \, \alpha_2 \;\; ,  \qquad \qquad  2 \, \varepsilon_3 \, = \, \alpha_3  \cr
   \varepsilon_1 - \varepsilon_2 - \varepsilon_3 \, = \, \alpha_1 \; ,  \quad  \varepsilon_1 + \varepsilon_2 - \varepsilon_3 \, = \, \alpha_1 + \alpha_2 \; ,   \quad  \varepsilon_1 - \varepsilon_2 + \varepsilon_3 \, = \, \alpha_1 + \alpha_3 \; ,  \quad  \varepsilon_1 + \varepsilon_2 + \varepsilon_3 \, = \, \alpha_1 + \alpha_2 + \alpha_3  }  $$
(those in first line being even, the others odd), and denote their set by  $ \, \Delta^+ \, $:  so
  $$  \Delta^+  \, = \,  \{ \alpha_1 \, , \, \alpha_2 \, , \, \alpha_3 \, , \, \alpha_1 + \alpha_2 \, , \, \alpha_1 + \alpha_3 \, , \, \alpha_1 + \alpha_2 + \alpha_3 \, , \, 2 \, \alpha_1 + \alpha_2 + \alpha_3 \}  $$
We call instead  {\sl negative\/}  the roots in  $ \, \Delta^- := -\Delta^+ \, $.  So the root system is given by  $ \; \Delta = \Delta^+ \coprod \Delta^- \; $.  We also set  $ \; \Delta_0^\pm := \Delta_0 \cap \Delta^\pm \; $,  $ \; \Delta_1^\pm := \Delta_1 \! \cap \Delta^\pm \; $.

\smallskip

   It is worth stressing at this point that the coroots  $ \, H_\alpha \in \fh \, $  associated to the positive roots are  $ \; H_{2\varepsilon_1} = {(1+a)}^{-1} \big( 2 \, h_1 - h_2 - a \, h_3 \big) \, $,  $ \; H_{2\varepsilon_2} = h_2 \; $,  $ \; H_{2\varepsilon_3} = h_3 \; $,  $ \; H_{\varepsilon_1 - \varepsilon_2 - \varepsilon_3} = h_1 \; $,  $ \; H_{\varepsilon_1 + \varepsilon_2 - \varepsilon_3} = h_1 - h_2 \; $,  $ \; H_{\varepsilon_1 - \varepsilon_2 + \varepsilon_3} = h_1 - a\,h_3 \; $,  and  $ \; H_{\varepsilon_1 + \varepsilon_2 + \varepsilon_3} = h_1 - h_2 - a\,h_3 \; $;  \; then the formula  $ \; H_{-\alpha} = H_\alpha \; $  yields coroots associated to negative roots out of those associated to positive ones.

\vskip7pt

   Now we introduce the following elements:
 \vskip-7pt
  $$  \begin{matrix}
   e_{1,2} := \big[ e_1 , e_2 \big] \; ,  &  e_{1,3} := \big[ e_1 , e_3 \big] \; ,  &  e_{1,2,3} := \big[ e_{1,2} , e_3 \big] \; ,  &  e'_{1,1,2,3} := \big[ e_1 , e_{1,2,3} \big]  \phantom{\Big|}  \\
   f_{2,1} := \big[ f_2 , f_1 \big] \; ,  &  f_{3,1} := \big[ f_3 , f_1 \big] \; ,  &  f_{3,2,1} := \big[ f_3, f_{2,1} \big] \; ,  &  f^{\,\prime}_{3,2,1,1} := \big[ f_{3,2,1} , f_1 \big]  \phantom{\Big|}
      \end{matrix}  $$
 \vskip1pt
\noindent
 All these are root vectors, respectively for the positive roots  $ \, \alpha_1 + \alpha_2 \, $,  $ \, \alpha_1 + \alpha_3 \, $,  $ \, \alpha_1 + \alpha_2 + \alpha_3 \, $  and  $ \, 2\,\alpha_1 + \alpha_2 + \alpha_3 \, $ (for the first line vectors) and similarly for the negative roots (for the second line vectors).  Moreover, by definition the generators  $ e_i $  and  $ f_i $  ($ \, i = 1,2,3 \, $)  are root vectors respectively for the roots  $ \, +\alpha_i \, $  and  $ \, -\alpha_i \, $  ($ \, i = 1,2,3 \, $).  As  $ \, \{h_1,h_2,h_3\} \, $  is a  $ \KK $--basis  of  $ \fh \, $,  we conclude that all these root vectors along with  $ h_1 \, $,  $ h_2 $  and  $ h_3 $  form a  $ \KK $--basis  of  $ \fg \, $.

\vskip4pt

   The relevant new brackets among all these basis elements   --- dropping the zero ones, those given by the very definition of  $ D(2,1;a) $,  those coming from others by (super-)skewcommuta\-tivity, and those involving the  $ h_i $'s  (given by the fact that all involved vectors are  $ \fh $--eigenvectors)  ---   are
  $$  \displaylines{
   \quad   \big[ e_1 , e_2 \big] = e_{1,2} \; ,  \quad  \big[ e_1 , e_3 \big] = e_{1,3} \; ,  \quad  \big[ e_1 , e_{1,2,3} \big] = e'_{1,1,2,3}   \hfill  \cr
   \quad   \big[ e_1 , f_{2,1} \big] = f_2 \; ,  \quad  \big[ e_1 , f_{3,1} \big] = a f_3 \; ,  \quad   \big[ e_1 , f^{\,\prime}_{3,2,1,1} \big] = -(1\!+a) f_{3,2,1}  \cr
   \quad   \big[ e_2 , e_{1,3} \big] = -e_{1,2,3} \; ,  \quad  \big[ e_2 , f_{2,1} \big] = f_1 \; ,  \quad  \big[ e_2 , f_{3,2,1} \big] = f_{3,1}   \phantom{\Big|}   \hfill  \cr
   \quad   \big[ e_3 , e_{1,2} \big] = -e_{1,2,3} \; ,  \quad  \big[ e_3 , f_{3,1} \big] = f_1 \; ,  \quad  \big[ e_3 , f_{3,2,1} \big] = f_{2,1}   \phantom{\Big|}   \hfill  \cr
   \quad   \big[ f_1 , f_2 \big] = -f_{2,1} \; ,  \quad  \big[ f_1 , f_3 \big] = -f_{3,1} \; ,  \quad  \big[ f_1 , f_{3,2,1} \big] = f^{\,\prime}_{3,2,1,1}   \hfill  \cr
   \quad   \big[ f_1 , e_{1,2} \big] = e_2 \; ,  \quad  \big[ f_1 , e_{1,3} \big] = a \, e_3 \; ,  \quad   \big[ f_1 , e'_{1,1,2,3} \big] = (1\!+a) \, e_{1,2,3}  \cr
   \quad   \big[ f_2 , f_{3,1} \big] = f_{3,2,1} \; ,  \quad  \big[ f_2 , e_{1,2} \big] = -e_1 \; ,  \quad  \big[ f_2 , e_{1,2,3} \big] = -e_{1,3}   \phantom{\Big|}   \hfill  \cr
   \quad   \big[ f_3 , f_{2,1} \big] = f_{3,2,1} \; ,  \quad  \big[ f_3 , e_{1,3} \big] = -e_1 \; ,  \quad  \big[ f_3 , e_{1,2,3} \big] = -e_{1,2}   \phantom{\Big|}   \hfill  \cr
   \quad   \big[ e_{1,2} , e_{1,3} \big] = -e'_{1,1,2,3} \;\; ,  \qquad  \big[ e_{1,2} , f_{2,1} \big] \, = \, h_1 \! - \! h_2 \;\; ,   \phantom{\Big|}   \hfill  \cr
   \hfill   \big[ e_{1,2} , f_{3,2,1} \big] = a f_3 \;\; ,  \qquad  \big[ e_{1,2} , f^{\,\prime}_{3,2,1,1} \big] = (1\!+a) f_{3,1}  \qquad  \phantom{\Big|}  \qquad
%%
%  \cr
%%
 }  $$
 \eject
  $$  \displaylines{
   \quad   \big[ e_{1,3} , f_{3,1} \big] \, = \, h_1 \! - \! a h_3 \; ,  \quad  \big[ e_{1,3} , f_{3,2,1} \big] = f_2 \; ,  \quad  \big[ e_{1,3} , f^{\,\prime}_{3,2,1,1} \big] = (1\!+a) f_{2,1}   \phantom{\Big|}   \hfill  \cr
   \quad   \big[ f_{2,1} , f_{3,1} \big] = -\!f^{\,\prime}_{3,2,1,1} \; ,  \quad  \big[ f_{2,1} , e_{1,2,3} \big] = a \, e_3 \; ,  \quad  \big[ f_{2,1} , e'_{1,1,2,3} \big] = -(1\!+a) \, e_{1,3}   \phantom{\Big|}   \hfill  \cr
   \quad   \big[ f_{3,1} , e_{1,2,3} \big] = e_2 \; ,  \quad  \big[ f_{3,1} , e'_{1,1,2,3} \big] = -(1\!+\!a) \, e_{1,2}   \phantom{\Big|}   \hfill  \cr
   \quad   \big[ e_{1,2,3} , f_{3,2,1} \big] = h_1 \! - \! h_2 \! - \! a h_3 \; ,  \quad  \big[ e_{1,2,3} , f^{\,\prime}_{3,2,1,1} \big] = -(1\!+a) f_1 \;\; ,   \phantom{\Big|}   \hfill  \cr
   \hfill   \big[ f_{3,2,1} , e'_{1,1,2,3} \big] = -(1\!+a) \, e_1 \; ,  \quad\!  \big[ e'_{1,1,2,3} , f^{\,\prime}_{3,2,1,1} \big] = -(1\!+a) \big( 2 \, h_1 \! - \! h_2 \! - \! a h_3 \big)  \cr
   \quad  e_{1,2}^{\,\langle 2 \rangle} = \, 0 \;\; ,  \quad  e_{1,3}^{\,\langle 2 \rangle} = \, 0 \;\; ,  \quad  e_{1,2,3}^{\,\langle 2 \rangle} = \, 0 \;\; ,   \qquad   f_{2,1}^{\,\langle 2 \rangle} = \, 0 \;\; ,  \quad  f_{3,1}^{\,\langle 2 \rangle} = \, 0 \;\; ,  \quad  f_{3,2,1}^{\,\langle 2 \rangle} = \, 0  }  $$
Now we modify just two root vectors taking
  $$  e_{1,1,2,3} \, := \, + {(1\!+a)}^{-1} \, e'_{1,1,2,3} \;\;\; ,  \qquad  f_{3,2,1,1} \, := \, - {(1\!+a)}^{-1} f^{\,\prime}_{3,2,1,1} \;\;\; ;  $$
(recall that  $ \, a \not= -1 \, $  by assumption); then the above formulas has to
be modified accordingly.

\medskip

   In particular now one checks that the even part of  $ \, \fg := D(2,1;a) \, $  is  $ \; \fg_0 = \rsl_2 \oplus \rsl_2 \oplus \rsl_2 \; $.  Moreover, the three triples
 $ \; \big( e_{1,1,2,3} \, , \, f_{3,2,1,1} \, , \, {(1\!+a)}^{-1} (2 \, h_1 - h_2 - a \, h_3) \big) \; $,  $ \; \big( e_2 \, , \, f_2 \, , \, h_2 \big) \; $,  $ \; \big( e_3 \, , \, f_3 \, , \, h_3 \big) \; $
are  $ \rsl_2 $--triples  inside  $ \fg \, $,  each one being associated to a (positive) even root  $ \, 2\,\varepsilon_i \, $  ($ \, i=1,2,3 \, $).

\bigskip

\section{Chevalley bases and Kostant superalgebras for  $ D(2,1;a) $}  \label{che-bas_kost_alg_D(2,1;a)}

\smallskip

   {\ } \quad   In this section we introduce the first results we shall build upon to construct Chevalley supergroups of type  $ D(2,1;a) \, $.  As before,  $ \KK $  is an algebraically closed field of characteristic zero.

\medskip

  \subsection{Chevalley bases and Chevalley Lie superalgebras}  \label{che-bas_alg}

\smallskip

   {\ } \quad   The subject of this subsection is an analogue, in the super setting, of a classical result due to Chevalley: the notion of Chevalley basis, and corrispondingly of Chevalley Lie superalgebra.  For  $ \, \fg := D(2,1;a) \, $,  \, this notion is introduced exactly like in  Definition 3.3 in  \cite{fg1},  {\sl up to changing  $ \Z $  to  $ \Z[a] \, $},  the latter being the unital subring of  $ \KK $  generated by  $ a $  (cf.~\S \ref{def_D(2,1;a)}).

\medskip

\begin{definition}  \label{def_che-bas}
 {\ }  We call  {\it Chevalley basis\/}  of  $ \; \fg := D(2,1;a) \, $  any homogeneous  $ \KK $--basis  $ \; B = {\big\{ H_i \big\}}_{1,2,3} \coprod {\big\{ X_\alpha \big\}}_{\alpha \in \Delta} \; $  of  $ \fg $ with the following properties:
 \vskip9pt
   \textit{(a)}  \quad   $ \, \big\{ H_1 , H_2 , H_3 \big\} \, $  is a  $ \KK $--basis  of  $ \fh \, $;  moreover, with  $ \, H_\alpha \! \in \fh \, $  as in  \S \ref{def_D(2,1;a)},

 \vskip4pt
   \centerline{ \qquad\qquad  $ \, \fh_{\Z[a]}  \; := \;   \text{\it Span}_{\,\Z[a]} \big( H_1 , H_2 , H_3 \big)  \; = \;  \text{\it Span}_{\,\Z[a]} \big( \{ H_\alpha \,|\, \alpha \! \in \! \Delta \}\big)  \quad $; }
 \vskip8pt
   \textit{(b)}  \hskip4pt   $ \big[ H_i \, , H_j \big] = 0 \, ,   \hskip9pt
 \big[ H_i \, , X_\alpha \big] = \, \alpha(H_i) \, X_\alpha \, ,   \hskip15pt  \forall \; i, j \! \in \! \{ 1, \dots, \ell \,\} \, ,  \; \alpha \! \in \! \Delta \; $;
 \vskip11pt
   \textit{(c)}  \hskip7pt   $ \big[ X_\alpha \, , \, X_{-\alpha} \big]  \, = \,  \sigma_\alpha \, H_\alpha  \hskip25pt  \forall \;\; \alpha \in \Delta $
 \vskip4pt
\noindent
 with  $ H_\alpha $  as in  {\it (a)},  and  $ \; \sigma_\alpha := -1 \; $  if  $ \, \alpha \in \Delta_1^- \, $,  $ \; \sigma_\alpha := 1 \; $  otherwise;
 \vskip13pt
   \textit{(d)}  \quad  $ \, \big[ X_\alpha \, , \, X_\beta \big]  \, = \, c_{\alpha,\beta} \; X_{\alpha + \beta}  \hskip17pt   \forall \;\, \alpha , \beta \in \Delta \, : \, \alpha + \beta \not= 0 \, $,  \, with
 \vskip9pt
   \hskip11pt   \textit{(d.1)} \,  if  $ \;\; (\alpha + \beta) \not\in \Delta \, $,  \; then  $ \;\; c_{\alpha,\beta} = 0 \; $,  \; and  $ \; X_{\alpha + \beta} := 0 \; $,
 \vskip7pt
   \hskip11pt   \textit{(d.2)} \,  if  $ \, \big( \alpha, \alpha \big) \not= 0 \, $  or  $ \, \big( \beta, \beta \big) \not= 0 \, $,  and if  $ \; \Sigma^\alpha_\beta := \big( \beta + \Z \, \alpha \big) \cap \big( \Delta \cup \{0\} \big) = \big\{ \beta - r \, \alpha \, , \, \dots \, , \, \beta + q \, \alpha \big\} \; $  is  {\sl the  $ \alpha $--string  through}  $ \beta \, $,  then  $ \; c_{\alpha,\beta} = \pm (r+1) \; $;
 \vskip5pt
   \hskip11pt   \textit{(d.3)} \,  if  $ \, \big( \alpha \, , \alpha \big) = \, 0 = \big( \beta \, , \beta \big) $,  \, then  $ \; c_{\alpha,\beta} = \pm \beta\big(H_\alpha\big) \; $.
\end{definition}

\smallskip

\begin{definition}
   If  $ B $  is a Chevalley basis of  $ \fg \, $,  we denote by  $ \, \fg^{\Z[a]} \, $  the  $ \Z[a] $--span of  $ B \, $,  and we call it the  {\it Chevalley Lie superalgebra\/}  (of  $ \fg $).
\end{definition}

\smallskip

\begin{remark}
 The above notions are taken from  \cite{fg1}:  in general, one should (and can) adapt them to the notion of ``Lie superalgebra'' in the stronger sense of  Definition \ref{def-Lie-salg}  (involving the ``2-ope-ration''): however, in the present case, i.e.~for  $ \, \fg = D(2,1;a) \, $,  no change is necessary.  Also, all examples of Chevalley bases considered in  \cite{fg1}, \S 3.3,  are Chevalley
 \hbox{bases in the ``stronger'' sense too.}
\end{remark}

\medskip

  \subsection{Existence of Chevalley bases}  \label{exist_Chev-bases}

\smallskip

%
%    {\ } \quad   The  {\sl existence\/}  of Chevalley bases (with slightly different definition)
% is proved in  \cite{ik};  we now present an explicit example of such a basis.
%
% \smallskip
%
   Let us consider in  $ \fg $  the generators  $ \, h_i \, $,  $ e_i \, $,  $ f_i \, $  ($ \, i=1,2,3 \, $)  and the root vectors  $ \, e_{1,2} \, $,  $ e_{1,3} \, $,  $ e_{1,2,3} \, $,  $ e_{1,1,2,3} \, $,  $ f_{2,1} \, $,  $ f_{3,1} \, $,  $ f_{3,2,1} \, $,  $ f_{3,2,1,1} \, $  constructed in  \S \ref{def_D(2,1;a)}. Looking at all brackets among them considered there, it is a routine matter to check that the set
  $$  B  \, := \,  {\big\{ H_i \, , \, e_i \, , \, f_i \big\}}_{i=1,2,3} \,{\textstyle \bigcup}\, \big\{ e_{1,2} \, , e_{1,3} \, , e_{1,2,3} \, , e_{1,1,2,3} \, , f_{2,1} \, , f_{3,1} \, , f_{3,2,1} \, , f_{3,2,1,1} \big\}  $$
with  $ \; H_1 := h_1 \, $,  $ \; H_2 := {(1\!+a)}^{-1} \big( 2 h_1 \! - \! h_2 \! - \! a h_3 \big) \, $,  $ \; H_3 := h_3 \; $,  \, is indeed a  {\sl Chevalley basis for}   $ \, \fg = D(2,1;a) \, $   --- the proof is a bookkeeping matter.
                                                                         \par
   To be precise, the root vectors in  $ B $  are:
  $$  \displaylines{
   \text{---  {\sl even\/}:}  \;\quad  X_{+(2 \, \alpha_1 + \alpha_2 + \alpha_3)} = e_{1,1,2,3} \; , \;\; X_{+\alpha_2} = e_2 \; , \;\; X_{+\alpha_3} = e_3   \qquad  \text{\sl (positive)}   \hfill  \cr
   \phantom{\text{---  {\sl even\/}:}}  \;\quad  X_{-(2 \, \alpha_1 + \alpha_2 + \alpha_3)} = f_{3,2,1,1} \; , \;\; X_{-\alpha_2} = f_2 \; , \;\; X_{-\alpha_3} = f_3   \qquad  \text{\sl (negative)}   \hfill  \cr
   \text{---  {\sl odd\/}:}  \;\quad  \text{\sl (positive)}  \qquad   X_{+\alpha_1} = e_1 \; , \;\; X_{+\alpha_1 + \alpha_2} = e_{1,2} \; , \;\; X_{+\alpha_1 + \alpha_3} = e_{1,3} \; , \;\; X_{+\alpha_1 + \alpha_2 + \alpha_3} = e_{1,2,3}  \cr
   \phantom{\text{---  {\sl odd\/}:}}  \phantom{\big|}  \!\quad  \text{\sl (negative)}  \qquad   X_{-\alpha_1} = f_1 \; , \;\; X_{-\alpha_1 - \alpha_2} = f_{2,1} \; , \;\; X_{-\alpha_1 - \alpha_3} = e_{3,1} \; , \;\; X_{-\alpha_1 - \alpha_2 - \alpha_3} = f_{3,2,1}  }  $$
while the Cartan elements in  $ B $  are just  $ \, H_1 \, $,  $ \, H_2 \, $  and  $ \, H_3 \, $  defined as above.
 See also  \cite{ik}.

\medskip

  \subsection{Kostant superalgebra}  \label{kost-form}

\smallskip

   {\ } \quad   For any  $ \KK $--algebra  $ A \, $,  given  $ \, n \in \N \, $  and  $ \, y \in A \, $  we define the  {\sl  $ n $--th  binomial coefficients}  $ \, {\Big(\! {y \atop n} \!\Big)} \, $  and the  {\sl  $ n $--th  divided power}  $ \, y^{(n)} \, $  by
%
%   $$  \bigg(\! {y \atop n} \!\bigg)  \, :=  \,  {\frac{\, y (y \! - \! 1) \cdots
% (y \! - \! n \! + \! 1) \,}{n!}}  \quad ,  \qquad  y^{(n)}  \, := \,  y^n \!\big/ n!  $$
%
 $ \,\; \displaystyle \bigg(\! {y \atop n} \!\bigg) \, := \, {\frac{\, y (y \! - \! 1) \cdots (y \! - \! n \! + \! 1) \,}{n!}} \;\, $,  $ \,\; \displaystyle y^{(n)} \, := \, y^n \!\big/ n! \;\, $.

\medskip

   Recall that  $ \, \Z[a] \, $  is the unital subring of  $ \KK $  generated by  $ a \, $.  We need also to consider the unital subring  $ \, \Z_a \, $  of  $ \KK $  generated by the subset $ \, \Big\{ \Big(\! {P(a) \atop n} \!\Big) \,\Big|\, P(a) \in \Z[a] \, , \, n \in \N \,\Big\} \; $.

  By a classical result on integer-valued polynomials (see  \cite{fg1},  Lemma 4.1) one shows that  $ \, \Z_a \, $  in fact is generated by  $ \, \Big\{ \Big(\! {a \atop n} \!\Big) \,\Big|\; n \in \N \,\Big\} \, $  as well.  Note also that  $ \, \Z_a = \Z \, $  if and only if  $ \, a \in \Z \, $.

\medskip

   Fix in  $ \, \fg := D(2,1;a) \, $  a Chevalley basis  $ \; B \, = \, \{ H_1 , H_2 , H_3 \} \, \coprod \, {\big\{ X_\alpha \big\}}_{\alpha \in \Delta} \, $  as in  Definition \ref{def_che-bas}.  Let  $ U(\fg) $  be the universal enveloping superalgebra of  $ \, \fg \, $.

\smallskip

   In   \cite{fg1},  \S 4.1, the Kostant superalgebra  $ K_\Z(\fg) $  was defined as the subalgebra of  $ U(\fg) $  generated by: divided powers of the root vectors attached to even roots, root vectors attached to odd roots, and binomial coefficients in the elements of  $ \fh_\Z \, $,  the  $ \Z $--span  of the elements of  $ \{H_1,H_2\,,H_3\} \, $.
                                                              \par
   If we try to perform the same construction  {\it verbatim\/}  for  $ \, \fg = D(2,1;a) \, $  for  $ \, a \in \KK \setminus \Z \, $,  we are soon forced to include among the generators all binomial coefficients of type  $ \, \Big(\! {H \atop n} \!\Big) \, $  with  $ \, H \in \fh_{\Z[a]} \, $,  the  $ \Z[a] $--span  of $ \{H_1\,,H_2\,,H_3\} \, $.  When commuting such binomial coefficients with divided powers, coefficients of type  $ \, \Big(\! {\alpha(H) \atop n} \!\Big) \, $  show up, where  $ \alpha $  is a root and  $ \, H \in \fh_{\Z[a]} \, $.  By construction we have  $ \, \alpha(H) \in \Z[a] \, $,  hence  $ \, \Big(\! {\alpha(H) \atop n} \!\Big) \, $  belongs to the ring  $ \, \Z_a \, $  defined above.

\medskip

   By the above remarks, for  $ \, \fg := D(2,1;a) \, $  we define the Kostant superalgebra  $ \, \kzag \, $  like we did in  \cite{fg1}  (for the other classical Lie superalgebras),  {\sl but\/}  with  $ \Z_a $
%
% replacing  $ \Z \, $
%
 as ground ring: more precisely,

\smallskip

\begin{definition}  {\;}
 We call  {\sl Kostant superalgebra},  or Kostant's  $ \Z_a $--form  of  $ U(\fg) \, $,  the unital  $ \Z_a $--subsuperalgebra  $ \kzag \, $  of  $ U(\fg) $  generated by all the elements
 \vskip-5pt
  $$  X_\alpha^{(n)} \! := X_\alpha^n \big/ n! \;\; ,  \quad  X_\gamma \;\; ,  \quad  \bigg({H \atop n}\bigg)   \qquad \hfill \forall \;\; \alpha \in \Delta_0 \; , \; n \in \N \, , \; \gamma \in \Delta_1 \, , \; H \in \fh_{\Z[a]}  $$
\end{definition}

\medskip

   The following analogue of  Corollary 4.2 in  \cite{fg1}  holds (with same proof), which is needed in the proof of the PBW-like theorem for  $ \kzag \, $:

\smallskip

\begin{corollary}  \label{int-polynom_kost}  {\ }
                                                    \par
    {\it (a)} \,  $ \mathbb{H}_{\Z_a} := \big\{\, h \! \in \! U(\fh) \;\big|\; h\big(z'_1,z'_2,z'_3\big) \in \Z_a , \,\; \forall \, z'_1, z'_2, z'_3 \in \Z_a \,\big\} \, $  is a  {\sl free}  $ \Z_a $--sub\-module  of  $ \, U(\fh) $,  with basis  $ \, B_{U(\fh)} := \Big\{ \Big(\! {H_1 \atop n_1} \!\Big) \Big(\! {H_2 \atop n_2} \!\Big) \Big(\! {H_3 \atop n_3} \!\Big) \,\Big|\, n_1, n_2, n_3 \! \in \! \N \Big\} \, $.
                                                    \par
   {\it (b)} \, The  $ \Z_a $--subalgebra  of  $ \, U(\fg) $  generated by all the elements  $ \, \Big(\! {{H' - z'} \atop n} \!\Big) \, $  with  $ \, H' \in \fh_{\Z_a} := \sum_{i=1}^3 \Z_a H_i \, $,  $ \, z' \in \Z_a \, $,  $ \, n \in \N \, $,  is nothing but  $ \, \mathbb{H}_{\Z_a} \, $.
\end{corollary}

\medskip

  \subsection{Commutation rules and Kostant's theorem}  \label{comm-rul_kost}

\smallskip

   {\ } \quad   In  \cite{fg1}  the authors proved a ``super PBW-like'' theorem for the Kostant's superalgebra: namely, the latter is a free  $ \Z $--module  with  $ \Z $--basis  the set of ordered monomials (w.~r.~to any total order) whose factors are binomial coefficients in the  $ H_i $'s,  or  {\sl odd\/}  root vectors, or divided powers of  {\sl even\/}  root vectors.  This result follows from a direct analysis of commutation rules among the generators of the Kostant's superalgebra.
%
%   The same procedure will work for  $ D(2,1;a) \, $.
%
% \vskip5pt
%
%    We list hereafter all commutation rules we need among generators of  $ \kzag \, $,
% all proved by easy induction; their most relevant feature is that all coefficients
% in these relations are in  $ \Z_a \, $.
%
 One can perform the same for  $ D(2,1;a) \, $,  using a list of relevant ``commutation rules'' (all proved by easy induction) whose main feature is that all coefficients belong to  $ \Z_a \, $.
%                                            \par
   We split the list into two sections:  {\it (1)\/}  relations involving only  {\sl even\/}  generators (known by classical theory);  {\it (2)\/}  relations involving also  {\sl odd\/}  generators.
 \vskip10pt

\noindent
 {\bf (1) Even generators only  {\rm  (that is  $ \, \Big(\! {H_i \atop m} \!\Big) $'s  and  $ \, X_\alpha^{(n)} $'s  only,  $ \, \alpha \in \Delta_0 \, $)}:}
 \vskip-7pt
  $$  \displaylines{
   \hfill   {\textstyle \Big({H_i \atop n}\Big)} \, {\textstyle \Big({H_j \atop m}\Big)}  \; = \;  {\textstyle \Big({H_j \atop m}\Big)} \, {\textstyle \Big({H_i \atop n}\Big)}  \hskip95pt
\forall \;\; i,j \in \{1,2,3\} \, , \;\;\; \forall \;\; n, m \in \N   \qquad  \cr
   \hfill   X_\alpha^{(n)} \, f(H)  \; = \;  f\big(H - n \; \alpha(H)\big) \, X_\alpha^{(n)}  \hfill
\forall \;\; \alpha \in \Delta_0 \, , \; H \in \fh \, , \; n \in \N \, , \; f(T) \in \KK[T]  \quad  \cr
   \hfill   X_\alpha^{(n)} \, X_\alpha^{(m)}  \; = \;  {\textstyle \Big(\! {{n \, + \, m} \atop m} \!\Big)} \, X_\alpha^{(n+m)}  \hskip93pt  \forall \; \alpha \in \Delta_0 \, , \;\; \forall \;\; n, m \in \N  \phantom{{}_{\big|}}   \hskip41pt  \cr
   \hfill   X_\alpha^{(n)} \, X_\beta^{(m)}  \; = \;  X_\beta^{(m)} \, X_\alpha^{(n)} \, + \, \textit{l.h.t}   \hskip83pt  \forall \; \alpha, \beta \in \Delta_0 \, , \;\; \forall \;\; n, m \in \N   \hskip39pt  }  $$
where  {\it l.h.t.}~stands for a  $ \Z_a $--linear  combination of monomials in the  $ X_\delta^{(k)} $'s  and in the $ \Big(\! {H_i \atop c} \!\Big) $'s  whose ``height''   --- by definition, the sum of all ``exponents''  $ k $ occurring in such a monomial ---   is less than  $ \, n+m \, $.  A  {\sl special case\/}  is
 \vskip-11pt
  $$  X_\alpha^{(n)} \, X_{-\alpha}^{(m)}  \,\; = \;\,  {\textstyle \sum_{k=0}^{min(m,n)}} \; X_{-\alpha}^{(m-k)} \, {\textstyle \Big( {{H_\alpha \, - \, m \, - \, n \, + \, 2 \, k} \atop k} \Big)} \, X_\alpha^{(n-k)}   \eqno   \forall \; \alpha \in \Delta_0 \, , \;\; \forall \;\; m, n \in \N   \quad  $$
 \vskip15pt

\noindent
 {\bf (2) Odd and even generators  {\rm  (also involving the  $ X_\gamma $'s,  $ \, \gamma \in \Delta_1 \, $)}:}
 \vskip-11pt
  $$  \displaylines{
   \hfill   X_\gamma \, f(H)  \; = \;  f\big(H - \gamma(H)\big) \, X_\gamma   \hskip65pt   \forall \;\; \gamma \in \Delta_1 \, , \; h \in \fh \, , \; f(T) \in \KK[T]   \qquad  \cr
   \hfill   X_\gamma^n \, = \, 0   \hskip135pt   \forall \;\; \gamma \in \Delta_1 \; ,  \hskip19pt  \forall \;\; n \geq 2  \qquad \qquad  \cr
   \hskip53pt   X_{-\gamma} \, X_\gamma  \; = \;  - X_\gamma \, X_{-\gamma} \, + \, H_\gamma   \hskip125pt  \forall \;\, \gamma \in \Delta_1   \qquad \qquad  \cr
   \hfill   X_\gamma \, X_\delta  \; = \;  - X_\delta \, X_\gamma \, + \, c_{\gamma,\delta} \, X_{\gamma + \delta}   \quad \hfill   \;\; \phantom{{}_{|}} \forall \;\, \gamma , \delta \in \Delta_1 \, , \; \gamma + \delta \not= 0   \hskip11pt \qquad  \cr
   \text{with  $ \, c_{\gamma,\delta} \, $  as in  Definition \ref{def_che-bas},}   \hfill  \cr
   \hfill   X_\alpha^{(n)} \, X_\gamma  \; = \;  X_\gamma \, X_\alpha^{(n)} \, + \, {\textstyle \sum_{k=1}^n} \! \, \Big( {\textstyle \prod_{s=1}^k} \, \varepsilon_s \Big) {\textstyle \Big(\! {{r \, + \, k} \atop k} \!\Big)} \, X_{\gamma + k \, \alpha} \, X_\alpha^{(n-k)}   \qquad  \forall \;\; n \in \N \, ,  \; \alpha \in \Delta_0 \, , \; \gamma \in \Delta_1   \phantom{{}_{\big|}}  \;\;
%
%%%
%
 }  $$
%
%%%
%      \cr
%    \hfill   X_\alpha^{(n)} \, X_\gamma  \; = \;  X_\gamma \, X_\alpha^{(n)} \, +
% \, {\textstyle \sum_{k=1}^n} \, \Big( {\textstyle \prod_{s=1}^k} \, \varepsilon_s \Big)
% \, {\textstyle \Big(\! {{r \, + \, k} \atop k} \!\Big)} \, X_{\gamma + k \, \alpha}
% \, X_\alpha^{(n-k)}   \hfill  \cr
%    \hfill   \forall \;\; n \in \N \, ,  \;\;\; \forall \;\; \alpha \in \Delta_0 \, ,
% \; \gamma \in \Delta_1   \phantom{{}_{\big|}}  \qquad  \cr
%
%%%
%
%%%
% %
%    \text{with} \hskip9pt  \sigma^\alpha_\gamma = \big\{ \gamma - r \, \alpha \, , \dots , \gamma \, ,
% \dots , \gamma + q \, \alpha \,\big\} \; ,  \hskip7pt  X_{\gamma + k \, \alpha} := 0
% \,\;\text{\ if\ }\;  (\gamma \! + \! k \, \alpha) \not\in \Delta \, ,   \hfill  \cr
%    \text{and}  \hskip11pt   \varepsilon_s = \pm 1  \text{\;\;\ such that \ }
% \big[ X_\alpha \, , X_{\gamma + (s-1) \, \alpha} \big] = \, \varepsilon_s \,
% (r+s) \, X_{\gamma + s \, \alpha}   \qquad  }  $$
%
%%%
%
 \vskip-7pt
\noindent
 with  $ \; \sigma^\alpha_\gamma = \big\{ \gamma - r \, \alpha \, , \dots , \gamma \, , \dots , \gamma + q \, \alpha \,\big\} \; $,  $ \; X_{\gamma + k \, \alpha} := 0 \; $  if  $ \; (\gamma \! + \! k \, \alpha) \not\in \Delta \, $,  and  $ \; \varepsilon_s = \pm 1 \; $  such that  $ \; \big[ X_\alpha \, , X_{\gamma + (s-1) \, \alpha} \big] = \, \varepsilon_s \, (r+s) \, X_{\gamma + s \, \alpha} \; $.

\vskip9pt

   Here now is our super-version of Kostant's theorem for  $ K_{\Z_a}(\fg) \, $:

\vskip13pt

\begin{theorem}  \label{PBW-Kost}
 The Kostant superalgebra  $ \, \kzag \, $  is a free\/  $ \Z_a $--module.  For any given total order  $ \, \preceq \, $  of the set  $ \, {\Delta} \cup \{1\,,2\,,3\} \, $,  a  $ \, \Z_a $--basis  of  $ \kzag $  is the set  $ {\mathcal B} $  of ordered ``PBW-like monomials'', i.e.~all products (without repeti\-tions) of factors of type  $ X_\alpha^{(n_\alpha)} $,  $ \Big(\! {H_i \atop n_i} \!\Big) $  and  $ \, X_\gamma $   --- with  $ \, \alpha \in \Delta_0 \, $,  $ \, i \in \{1\,,2\,,3\} \, $,  $ \, \gamma \in \Delta_1 \, $  and  $ \, n_\alpha $,  $ n_i  \in \N $  ---   taken in the right order with respect to  $ \, \preceq \; $.
\end{theorem}

\vskip9pt

   This result is proved like the similar one in  \cite{fg1},  making use of the commutation relations conside\-red above.  It also has a direct consequence, again proved like in  \cite {fg1}.  To state it, let first consider  $ \fg_1^{\Z[a]} \, $,  the odd part of  $ \, \fg^{\Z[a]} \, $:  it has  $ \, \big\{ X_\gamma \,\big|\, \gamma \! \in \! \Delta_1 \big\} \, $  as  $ \Z[a] $--basis,  by construction.  Then let  $ \, \bigwedge \fg_1^{\Z[a]} \, $  be the exterior\/  $ \Z[a] $--algebra  over  $ \fg_1^{\Z[a]} \, $,  and  $ \, \bigwedge \fg_1^{\Z_a} \, $  be its scalar extension to  $ \Z_a \, $.  Let also  $ K_\Z(\fg_0) $  be the classical Kostant's algebra of  $ \fg_0 $  (over  $ \Z $)  and let  $ K_{\Z_a}\!(\fg_0) $  be its scalar extension to  $ \Z_a \, $.  Then the tensor factorization  $ \; U(\fg) \, \cong \, U(\fg_0) \otimes_{\KK} \bigwedge \fg_1 \; $  (see  \cite{vsv})  has the following ``integral version'':

\vskip13pt

\begin{corollary}  \label{kost_tens-splitting} {\ }
 There exists an isomorphism of  $ \, \Z_a $--modules  $ \,\; \kzag \, \cong \, K_{\Z_a}\!(\fg_0) \otimes_{\,\Z_a} {\textstyle \bigwedge} \fg_1^{\Z_a} \;\, $.
%
% %
%  \vskip-11pt
% %
%   $$  \kzag  \,\; \cong \;\,
% K_{\Z_a}\!(\fg_0) \, \otimes_{\,\Z_a} {\textstyle \bigwedge} \, \fg_1^{\Z_a}  $$
%
%
\end{corollary}

\vskip7pt

\begin{remarks}
 {\it (a)}  Following a classical pattern (and  cf.~\cite{bk},  \cite{bku},  \cite{sw}  in the super context) we can define the  {\sl superalgebra of distributions\/}  $ \, {\mathcal D}\textit{ist}\,(G) \, $  on any supergroup  $ G \, $.
 Then  $ \; {\mathcal D}\textit{ist}\,(G) = \kzag \otimes_{\Z_a} \! \KK \; $,  when  $ \; \fg := \textit{Lie}\,(G) \; $  is just  $ \, D(2,1;a) \, $.
 \vskip2pt
   {\it (b)}  All the above proves that the assumptions of Theorem 2.8 in  \cite{sw}  hold for any supergroup  $ G $  with tangent Lie superalgebra   $ \, \fg = D(2,1;a) \, $.  Thus  {\sl all results in  \cite{sw}  do apply to such supergroups}.
\end{remarks}

\bigskip

\section{Chevalley supergroups of type  $ D(2,1;a) $}  \label{che-sgroup}

\smallskip

   {\ } \quad   I present now the construction of affine supergroups associated to the Lie superalgebra  $ \, \fg = D (2,1;a) \, $.  The method, inspired to Chevalley's original one (dealing with complex semisimple Lie algebras), follows closely the one presented in  \cite {fg1}  for the other classical Lie superalgebras   --- including  $ \, \fg = D (2,1;a) \, $  when  $ \, a \in \Z \, $.  However, the occurrence of the (possibly non-integral) parameter $ a $  demands to revisit that construction and to introduce some suitable, delicate modifications.

\medskip

\subsection{Admissible lattices}  \label{adm-lat}

\smallskip

   {\ } \quad   Let  $ \KK $  be an algebraically closed field of characteristic zero.  If  $ R $  is a uni\-tal subring of  $ \KK \, $,  and  $ V $  a finite dimensional  $ \KK $--vector  space, any  $ \, M \! \subseteq \! V \, $  is called  $ R $--{\it lattice\/}  (or  $ R $--form)  of  $ V $  if  $ \, M \! = {\textit{Span}}_R(B) \, $  for some  $ \KK $--basis  $ B $  of  $ V $.
%
%  \vskip4pt
%
   Let  $ \, \fg = D(2,1;a) \, $  be defined over  $ \KK \, $,  and fix the ring  $ \Z_a \, $,  a Chevalley basis  $ B $  of  $ \fg $  and the Kostant algebra  $ \kzag $  as in  \S \ref{che-bas_kost_alg_D(2,1;a)}.

\vskip5pt

   The following definition and results are just slight variations of those in  \cite{fg1},  \S 5.1.

\vskip9pt

\begin{definition}
 Let  $ V $  be a  $ \fg $--module,  and let  $ M $  be a  $ \Z_a $--lattice  of it.
 \vskip7pt
\noindent
 \hskip7pt  {\it (a)} \,  We call  $ V $  {\it rational\/}  if:
 \vskip3pt
 \hskip1pt  {\it (a-1)} \,  $ \, \fh_{\Z[a]} := \text{\it Span}_{\Z[a]} \big( H_1, H_2, H_3 \big) \, $  acts diagonally on  $ V $  with eigenvalues in  $ \Z[a] \, $;  in other words, one has  $ \, V \! = \! \bigoplus_{\mu \in \fh^*} \! V_\mu \, $,  with  $ \, V_\mu := \big\{ v \! \in \! V \,\big|\, h.v = \mu(h) \, v \; \forall \, h \! \in \! \fh \big\} \, $,  {\sl and}  $ \, \mu(H_i) \in \Z[a] \, $  for all  $ i $  and all  $ \mu $  such that  $ \, V_\mu \not= \{0\} \, $;
 \vskip3pt
 \hskip1pt  {\it (a-2)} \,  $ \, \mu(H_\alpha) \in \Z \, $  for all  $ \, \alpha \in \Delta_0 \, $  and  $ \, \mu \in \fh^* \, $  such that  $ \, V_\mu \not= \{0\} \, $.
 \vskip6pt
\noindent
 \hskip7pt  {\it (b)\/}  We call  $ M $  {\it admissible (lattice)\/}  if it is  $ \kzag $--stable.
\end{definition}

\vskip7pt

\begin{theorem}
  Any rational, finite dimensional, semisimple  $ \fg $--module  $ V $  contains an admissible lattice  $ M $.  Any such  $ M $  is the (direct) sum of its weight components,
%
% i.e.~$ \, M = \bigoplus\limits_{\mu \in \fh^*} M_\mu \, $  with  $ \, M_\mu = M \cap V_\mu \, $.
%
 i.e.~$ \, M = \bigoplus\limits_{\mu \in \fh^*} \big( M \cap V_\mu \big) \, $.
\end{theorem}

\vskip-3pt

\begin{theorem}  \label{stabilizer}
   Let  $ V $  be a rational, finite dimensional\/  $ \fg $--module,  $ M $  an admissible lattice of  $ \, V $,  and  $ \, \fg_V \! = \! \big\{ X \! \in \! \fg \,\big|\, X.M \subseteq \! M \big\} \, $.  If  $ \, V $  is faithful, then
 \vskip-8pt
  $$  \fg_V  \, = \,  \fh_V \; {\textstyle \bigoplus \big( \oplus_{\alpha \in \Delta}} \, \Z_a \, X_{\alpha} \big) \; ,  \;\quad  \fh_V \, := \, \big\{ H \in \fh \;\big|\; \mu(H) \in \Z_a \, , \; \forall \; \mu \in \Lambda \big\}  $$
 \vskip-1pt
\noindent
 where  $ \Lambda $  is the set of all weights of  $ \, V $.  In particular,  $ \fg_V $  is a  $ \Z_a $--lattice  in  $ \fg \, $,  independent of the choice of the admissible lattice  $ M $  (but not of  $ \, V $).
\end{theorem}

\medskip

   From now on, we retain the following notation:  $ V $  is a rational, finite dimensional  $ \fg $--module,  and  $ M $  is an admissible lattice of  $ V \, $.  Also,  {\sl we assume  $ \bk $  to be a commutative unital  $ \Z_a $--algebra}.
 \vskip4pt
   With these assumptions, we set
%
%   $$  \fg_\bk \, := \, \bk \otimes_{\Z_a} \! \fg_V \;\; , \qquad  V_\bk \, :=
% \, \bk \otimes_{\Z_a} \! M \;\; ,  \qquad  U_\bk(\fg) \, := \, \bk \otimes_{\Z_a} \! \kzag \;\; ;  $$
% %
%
 $ \; \fg_\bk \, := \, \bk \otimes_{\Z_a} \! \fg_V \; $,  $ \; V_\bk \, := \, \bk \otimes_{\Z_a} \! M \; $,  $ \; U_\bk(\fg) \, := \, \bk \otimes_{\Z_a} \! \kzag \; $;  \,
so  $ \fg_\bk $  acts faithfully on  $ V_\bk \, $,  yielding an embedding of  $ \fg_\bk $  into  $ \rgl(V_\bk) \, $.  For any  $ \, A \in \salg_\bk \, $,  the Lie superalgebra  $ \, \fg_A := A \otimes_\bk \fg_\bk \, $  acts faithfully on  $ \, V_\bk(A) := A \otimes_\bk V_\bk \, $,  so it embeds into  $ \rgl\big(V_\bk(A)\!\big) \, $,  etc.

\bigskip

\subsection{Additive one-parameter supersubgroups}  \label{add_one-param-ssgroups}

\vskip3pt

   {\ } \quad   Let  $ \, \alpha \in \Delta_0 \, $,  $ \, \beta \in \Delta_1 \, $,  and let  $ X_\alpha \, $,  $ X_\beta $  be the associated root vectors (in our fixed Chevalley basis of  $ \fg \, $).  Both  $ X_\alpha \, $  and  $ X_\beta $  act as nilpotent operators on  $ V \, $,  hence on  $ M \, $  and  $ V_\bk \, $,  so they are represented by nilpotent matrices in  $ \rgl(V_\bk(A)) \, $;  the same holds for all operators
  $$  t \, X_\alpha \; ,  \;\; \vartheta \, X_\beta \,\; \in \;
\End\big(V_\bk(A)\big)   \eqno \forall \;\; t \in A_0 \, ,  \; \vartheta \in A_1 \;\, .   \qquad (4.1)  $$
 Of course we have  $ \, Y^{(n)} := Y^n \! \big/ n! \in \big(\kzg\big)(A) \, $  for any  $ Y $  as in (4.1) and  $ \, n \in \N \, $;  moreover,  $ \, Y^{(n)} = 0 \, $  for  $ \, n \gg 0 \, $,  by nilpotency.  Thus the formal power series  $ \; \exp(Y) := \sum_{n=0}^{+\infty} \, Y^{(n)} \; $,  when computed for  $ Y $  as in (4.1), gives a well-defined element in  $ \rGL\big(V_\bk(A)\big) \, $,  expressed as finite sum.

\begin{definition}  \label{def_add_one-par-ssgrp}
 \;  Let  $ \, \alpha \! \in \! \Delta_0 \, $,  $ \, \beta \! \in \! \Delta_1 \, $,  and  $ X_\alpha \, $,  $ X_\beta \, $  as above.  We define the supergroup functors  $ x_\alpha \, $  and  $ x_\beta \, $  from  $ \salg $  to  $ \grps $  as
  $$  \begin{array}{rl}
   x_\alpha(A) \!\!\!  &  :=  \big\{ x_\alpha(t) \! := \exp\!\big( t X_\alpha \big) \big|\, t \! \in \! A_0 \big\}  \, = \,  \big\{ \big( 1 + t \, X_\alpha + {t \over 2!} \, X_\alpha^{\,2} + \cdots \big) \big|\, t \! \in \! A_0 \big\}  \\
  \\
   x_\beta(A) \!\!\!  &  :=  \big\{ x_\beta(\vartheta) \! := \exp\!\big( \vartheta X_\beta \big) \,\big|\, \vartheta \in A_1 \big\}  \, = \,  \big\{ \big( 1 +  \vartheta \, X_\beta \big) \,\big|\, \vartheta \in A_1 \,\big\}
      \end{array}  $$
   \indent   For later convenience we also write  $ \, x_\zeta(\bt) := 1 \, $  when  $ \zeta  $  belongs to the  $ \Z $--span  of  $ \Delta $  but  $ \, \zeta \not\in \Delta \, $.
\end{definition}

\vskip5pt

   Like in  \cite{fg1},  Proposition 5.8{\it (a)},  one sees that these supergroup functors are in fact representable, hence they both are  {\sl affine supergroups\/}:  namely,  $ x_\alpha $  is represented by  $ \bk[x] $  and  $ x_\beta $  by  $ \bk[\xi] \, $.
%
% ,  these being Hopf algebras with coproduct  $ \, \varDelta(x) = x \otimes 1 + 1 \otimes x \, $  and
% $ \, \varDelta(\xi) = \xi \otimes 1 + 1 \otimes \xi \, $.
%
 We shall refer to both  $ x_\alpha $  and  $ x_\beta $  as  {\sl additive one-parameter (super)subgroups}.

\medskip

\subsection{Multiplicative one-parameter supersubgroups of
\allowbreak classical type}  \label{mult_one-param-ssgroups_cl-type}

\smallskip

   {\ } \quad   For any  $ \, \alpha \in \Delta_0 \, \big(\! \subseteq \fh^* \big) \, $,  let  $ \, H_\alpha \in \fh^0_\Z \, $  be the corresponding coroot  (cf.~\S \ref{def_D(2,1;a)}).  Then consider  $ \, \fh^0_\Z := \text{\it Span}_{\,\Z} \big( \big\{ H_\alpha \,\big|\, \alpha \! \in \! \Delta_0 \big\}\big) \, $;  clearly this is a  $ \Z $--form  of  $ \fh \, $,  and by definition we have  $ \, \fh^0_\Z \subseteq \fh_{\Z[a]} := \text{\it Span}_{\,\Z[a]} \big( H_1, H_2, H_3 \big) \, $.  Let  $ \, V = \oplus_\mu V_\mu \, $  be the splitting of  $ V $  into weight spaces; as  $ V $  is rational, we have  $ \; \mu(H_\alpha) \in \Z \; $  for all  $ \, \alpha \in \Delta_0 \, $  and  $ \, \mu \in \fh^* : V_\mu \not= \{0\} \, $.  Now, for any  $ \, A \in \salg \, $,  $ \, \alpha \in \Delta_0 \, $  and  $ \, t \in U(A_0) \, $  (the group of invertible elements in  $ A_0 \, $)  set  $ \; h_\alpha(t).v  \, := \,  t^{\mu(H_\alpha)} \, v  \; $  for all  $ \, v \in V_\mu \, $,  $ \, \mu \in \fh^* \, $:  this defines another operator (also locally expressed by exponentials)
 \vskip-7pt
  $$  h_\alpha(t) \, \in \, \rGL\big(V_\bk(A)\big)   \eqno \forall \;\; t \in U(A_0) \, , \; \alpha \in \Delta_0 \;\; .  \qquad (4.2)  $$

More in general, if  $ \, \big\{ H_{\alpha_i} \big\}_{i=1,2,3} \, $  is any basis  of  $ \Delta_0 $  and  $ \, H = \sum_{i=1}^3 z_i H_{\alpha_i} \, $  (with  $ z_1, z_2, z_3 \in \Z \, $)  then we define  $ \; h_H(t) := \prod_{i=1}^3 {\big( h_{\alpha_i}(t) \big)}^{z_i} \; $,  \, for  $ \, \alpha \in \Delta_0 \, $.

\vskip11pt

\begin{definition}  \label{def_mult_one-par-ssgrp_cl-type}
 \;  Let  $ \, H \in \fh^0_\Z \, $  as above.  We define the supergroup functor  $ h_H \, $  (also writing  $ \; h_\alpha := h_{H_\alpha} \; $  for any  $ \, \alpha \in \Delta_0 \, $)  from  $ \salg $  to  $ \grps $  as
%
%   $$  A  \,\; \mapsto \;\,  h_H(A)  \; := \;  \big\{\, t^H := h_H(t) \;\big|\; t \in \! U(A_0) \big\}  $$
% %
 $ \; A \, \mapsto \, h_H(A) \, := \, \big\{\, t^H := h_H(t) \;\big|\; t \in \! U(A_0) \big\} \; $.
 Like in  \cite{fg1},  Proposition 5.8{\it (b)},  one sees that these functors are represen\-table, so they are affine supergroups; even more, they are also closed subgroups of the diagonal subgroup of  $ \rGL\big(V_\Bbbk(A)\big) \, $.
%
%%%%%
%
% $ h_H $  is represented by  $ \bk\big[z,z^{-1}\big] \, $,  a Hopf algebra with coproduct
% $ \, \varDelta\big(z^{\pm 1}\big) = z^{\pm 1} \otimes z^{\pm 1} \, $.  We shall then refer
% to these as  {\sl multiplicative one-parameter (super)subgroups of classical type}.
%
%%%%%
%
\end{definition}

\medskip

\subsection{Multiplicative one-parameter supersubgroups of
\allowbreak $ a $--type}  \label{mult_one-param-ssgroups_a-type}

\smallskip

   {\ } \quad   In order to attach a suitable ``multiplicative one-parameter supersubgroup'' to  {\sl any\/}  element in  $ \, \fh_{\Z[a]} := \text{\it Span}_{\,\Z[a]} \big( H_1, H_2, H_3 \big) \, $   --- not only in  $ \fh^0_\Z $  ---   we need to adapt our previous construction.

\vskip4pt

   Consider a Cartan element  $ H_i $  in our fixed Chevalley basis: we want to define a suitable, representable supergroup functor associated to it, to be called  $ \, h^{[a]}_i \, $.  Given any  $ \, A \in \salg \, $,  and  $ \, t \in U(A_0) \, $,  we look for an operator like  $ \; h^{[a]}_i(t) := t^{H_i} \in \rGL\big(V_\bk(A)\big) \; $.  This should be given by
  $$  h^{[a]}_i(t)  \; := \;  t^{H_i}  \; = \;  {\big( 1 + (t \! - \! 1) \big)}^{H_i}  \; = \;\,  {\textstyle \sum_{n=0}^{+\infty}} \, {(t \! - \! 1)}^n \, {\textstyle \Big(\! {H_i \atop n} \!\Big)}  $$
   Let  $ \, V = \bigoplus_{\mu \in \fh^*} \! V_\mu \, $  be the splitting of  $ V $  into weight spaces; definitions imply
  $$  h^{[a]}_i(t)\Big|_{V_\mu}  \, = \;\,  {\textstyle \sum_{n=0}^{+\infty}} \, {(t \! - \! 1)}^n \, {\textstyle \Big(\! {\mu(H_i) \atop n} \!\Big)} \, \text{\sl id}_{V_\mu}  \; = \;\,  t^{\mu(H_i)} \, \text{\sl id}_{V_\mu}  $$
on weight spaces, which makes sense   --- and then globally yields a well-defined operator on all of  $ \, V = \bigoplus_{\mu \in \fh^*} \! V_\mu \, $  ---   as soon as  $ \; t^{\mu(H_i)} := \sum_{n=0}^{+\infty} \, {(t \! - \! 1)}^n \,\Big(\! {\mu(H_i) \atop n} \!\Big) \; $  is a well-defined element of  $ A \, $.  Now,  $ V $  is rational and  $ M $  is admissible,  so  $ \, \mu(H_i) \in \Z[a] \, $.  Thus a necessary condition we may require is
  $$  t^{z(a)}  \, := \;  {\textstyle \sum_{n=0}^{+\infty}} \, {(t \! - \! 1)}^n \, {\textstyle \Big(\! {z(a) \atop n} \!\Big)}  \, \in \,  A   \qquad  \forall \;\; z(a) \in \Z[a]   \eqno (4.3)  $$
which in the end is equivalent to having
  $$  t^{\pm a^k}  \, := \;  {\textstyle \sum_{n=0}^{+\infty}} \, {(t \! - \! 1)}^n \, {\textstyle \Big(\! {{\pm a^k} \atop n} \!\Big)}  \, \in \,  A   \qquad  \forall \;\; k \in \N   \eqno (4.4)  $$
Both (4.3) and (4.4) must be read as conditions defining a suitable subset of  $ A_0 \, $,  namely that of all elements  $ \, t \in A_0 \, $  for which the condition does hold.
%                                                 \par
   Now we go and fix details, as to give a well-defined meaning to the expressions in (4.3--4) and to the just sketched construction.

\smallskip

   Consider the polynomial  $ \Z_a $--superalgebra  $ \, \Z_a[\,\ell\,] \, $,  where  $ \ell $  is an  {\sl even\/}  indeterminate: this is also a (super)bialgebra, with  $ \, \varDelta(\ell) = \ell \otimes \ell \, $,  $ \, \epsilon(\ell) = 1 \, $.  Let  $ \, \Z_a[[\,\ell-\!1]] \, $  be the  $ (\ell\!-\!1) $--adic  completion of  $ \Z_a[\,\ell\,] \, $;  this contains also an inverse of  $ \ell \, $,  namely  $ \, \ell^{-1} = \sum_{n=0}^{+\infty} \, {(-1)}^n \, {(\ell \! - \! 1)}^n \, $,  so  $ \; \Z_a\big[\, \ell \, , \ell^{-1} \big] \subseteq \Z_a[[\,\ell-\!1]] \; $.  The coproduct of  $ \Z_a[\,\ell\,] $  extends to  $ \Z_a[[\,\ell-\!1]] $  making it into a  {\sl formal\/} bialgebra   --- the coproduct taking values into the  $ (\ell\!-\!1) $--adic  completion of the algebraic tensor square of  $ \Z_a[[\,\ell-\!1]] $  ---   which indeed is a  {\sl formal Hopf algebra}.  In the latter, both  $ \ell $  and  $ \ell^{-1} $  are group-like elements, so that $ \Z_a\big[\, \ell \, , \ell^{-1} \big] $  is a (non formal) Hopf subalgebra of  $ \Z_a[[\,\ell\!-\!1]] \, $.  Consider the elements
  $$  {} \qquad   \ell^{\,\pm a^k}  \, := \,  {\textstyle \sum_{n=0}^{+\infty}} \, {(\, \ell - \! 1)}^n \, {\textstyle \Big(\! {{\pm a^k} \atop n} \!\Big)}  \; \in \;  \Z_a[[\,\ell\!-\!1]]   \eqno \forall \;\; k \in \N  \qquad  (4.5)  $$
 A key property of these elements is the following:

\vskip11pt

\begin{lemma}  \label{l^a_group-like}
 All the  $ \ell^{\,\pm a^k} $'s  ($ \, k \in \N $)  are group-like elements of  $ \, \Z_a[[\,\ell\!-\!1]] \, $.
\end{lemma}

\begin{proof}
 When  $ \ell $  belongs to  $ \mathbb{C} $  and  $ \, a \in \mathbb{C} \setminus \{0,-1\} \, $,  the power series  $ \; {\textstyle \sum\limits_{n=0}^{+\infty}} {(\ell \! - \! 1)}^n \, {\textstyle \Big(\! {{\pm a^k} \atop n} \!\Big)} \; $  represents the Taylor expansion of the analytic function  $ \; \phi_a \! : \ell \mapsto \ell^{\,\pm a^k} \; $  in a neighbourhood of  $ \ell_0 := 1 \, $.  Now, the function  $ \phi_a $  is multiplicative, i.e.~$ \; \phi_a\big( \ell_1 \, \ell_2 \big) = \phi_a(\ell_1) \, \phi_a(\ell_2) \; $;  this identity for all complex values of $ \ell_1 $  and  $ \ell_2 $  in a neighbourhood of 1 implies (passing through Taylor expansion) a similar identity at the level of power series.  In turn, the latter identity implies an identity among  {\sl formal power series\/}  (i.e., still holding when complex numbers  $ \ell_1 $,  $ \ell_2 $  are replaced with indeterminates).  This can be recast saying that the formal power series  $ \, {\textstyle \sum_{n=0}^{+\infty}} \, {(\, \ell \! - \! 1)}^n \, {\textstyle \Big(\! {{\pm a^k} \atop n} \!\Big)} \, $  is a group-like element in  $ \, \Z_a[[\,\ell\!-\!1]] \, $.  As this holds for  {\sl any}  $ \, a \in \mathbb{C} \, $,  it must hold for a formal parameter  $ a \, $,  i.e.~if the complex value of  $ a $  is replaced by an indeter\-minate.  Then the formal symbol  $ a $  can be replaced by any genuine value in  $ \, \KK \setminus \{0,-1\} \, $,  and the formal series in  $ \, \Z_a[[\,\ell\!-\!1]] \, $  will still be group-like.
\end{proof}

\vskip9pt

\begin{free text}  \label{aff_sgrps-P_a}
 {\bf The affine (super)groups  $ P_a \, $.}  Let  $ \, \cL_a := \Z_a\big[ \big\{\ell^{\,\pm a^k}\big\}_{k \in \N} \,\big] \, $  be the  $ \Z_a $--subalgebra  of  $ \, \Z_a[[\,\ell\!-\!1]] \, $  generated by all the  $ \ell^{\,\pm a^k} $'s:  as these generators are group-like,  $ \Z_a\big[ \big\{\ell^{\,\pm a^k}\big\}_{k \in \N} \,\big] $  is in fact a Hopf sub(super)algebra of  $ \, \Z_a[[\,\ell\!-\!1]] \, $.  In particular, it is a (totally even) Hopf algebra over  $ \Z_a \, $.  If now  $ \bk $  is any (unital, commutative)  $ \Z_a $--algebra  as above, then  $ \; \cL_{a,\bk} := \bk \otimes_{\Z_a} \Z_a\big[ \big\{\ell^{\,\pm a^k}\big\}_{k \in \N} \,\big] =: \bk\big[ \big\{\ell^{\,\pm a^k}\big\}_{k \in \N} \,\big] \;\, \big( \subseteq \bk \otimes_{\Z_a} \Z_a[[\,\ell\!-\!1]] \,\big) \, $  belongs to  $ \, \salg_\bk = \salg \, $,  and in addition it is a (commutative) Hopf algebra over  $ \bk \, $.
                                                         \par
   Let  $ \, P_a := \uspec\big(\cL_{a,\bk}\big) \, $  be the affine scheme associated to  $ \cL_{a,\bk} \, $:  as  $ \cL_{a,\bk} $  is a Hopf algebra,  $ P_a $  is indeed an affine group-scheme, which we can also see as a (affine) supergroup.  As usual, we shall identify  $ \, P_a := \uspec\big(\cL_{a,\bk}\big) \, $  with its functors of points, namely  $ \, P_a = \Hom_{\salg_\bk}\big( \cL_{a,\bk} \, , \text{---} \,\big) \, $.  Similarly, we identify  $ \, U := \uspec\big(\bk\big[ \ell \, , \ell^{-1} \big] \big) \, $  with its functor of points, so that for any  $ \, A \in \salg_\bk \, $  we have that  $ \, U(A) = \Hom_{\salg_\bk}\big( \bk\big[ \ell \, , \ell^{-1} \big] \, , A \big) \, $  is the set of units of  $ A_0 \, $.  Then the natural embedding of  $ \, \bk\big[ \ell \, , \ell^{-1} \big] \, $  into  $ \, \cL_{a,\bk} = \bk\big[ \big\{\ell^{\,\pm a^k}\big\}_{k \in \N} \,\big] \, $  yields a (super)group morphism  $ \, \pi_a : P_a \! \longrightarrow U \, $.
                                                          \par
   Given  $ \, A \in \salg_\bk \, $,  any  $ \, \varphi \in P_a(A) = \Hom_{\salg_\bk}\big( \cL_{a,\bk} \, , A \big) \, $   --- an algebra morphism from  $ \cL_{a,\bk} $  to  $ A $  ---   is uniquely determined by a double sequence in  $ \in A \, $,  namely  $ \, \undt := {\big( t_k^+ \, , t_k^- \big)}_{k \in \N} \, $  with  $ \, t_k^\pm := \varphi\big(\ell^{\pm a^k}\big) \, $;  as  $ \, t_k^\mp = {\big( t_k^\pm \big)}^{-1} \, $,  just one half of this double sequence is actually enough to determine  $ \varphi \, $.  Similarly, any  $ \, \phi \in U(A) \, $  is uniquely determined by  $ \, t := \phi(\ell) \, $.  Via these identifica\-tions, the morphism  $ \, \pi_a : P_a \! \longrightarrow U \, $  is described by  $ \, \pi_a\big( {\big( t_k^+ \, , t_k^- \big)}_{k \in \N} \big) = t_0^+ \; $.
                                                          \par
   Note that when  $ \, a \in \Z \, $  one has  $ \, P_a = U \, $,  by the very definitions, and  $ \pi_a $  is the identity.

\vskip5pt

   {\sl $ \underline{\text{N.B.}} $:}  In some cases,  $ \, \varphi \in P_a(A) \, $  is uniquely determined by its image  $ \, \pi_a(\varphi) \, $  in  $ U(A) \, $;  under mild assumptions, this may happen for all elements of  $ P_a(A) \, $  so that  $ \pi_a $  turns out to be injective and  $ P_a(A) $  identifies with a subgroup of  $ U(A) \, $.  Indeed, given  $ \, \varphi \in P_a(A) \, $  and  $ \, t := \pi_a(\varphi) \in U(A) \, $,  let  $ \, \widehat{A} \, $  be the  $ (t-\!1) \, $--adic  completion of  $ A \, $,  and let  $ \, A \,{\buildrel {j_t} \over \longrightarrow}\, \widehat{A} \, $  be the natural morphism from  $ A $  to  $ \widehat{A} \, $.  Then there exists a unique morphism  $ \, \varphi': \bk[\,\ell-1] \longrightarrow A \, $  given by  $ \, \varphi'(\,\ell-1) = t-1 \, $;  in turn, this uniquely yields  $ \, \widehat{\varphi}: \bk[[\,\ell-1]] \longrightarrow \widehat{A} \, $  such that  $ \, \widehat{\varphi}(\,\ell-1) = t-1 \, $:  finally, the restriction of  $ \widehat{\varphi} $  to  $ \cL_{a,\bk} $  necessarily coincides with the composition  $ \; j_t \circ \varphi : \cL_{a,\bk} \,{\buildrel {\varphi\,} \over \longrightarrow}\, A \,{\buildrel {j_t} \over \longrightarrow}\, \widehat{A} \;\, $.  Thus  $ \, t = \pi_a(\varphi) \, $  uniquely determines  $ \, j_t \circ \varphi \, $:  in particular, if  $ \, \varphi, \psi \in P_a(A) \, $  and  $ \, \pi_a(\varphi) = t = \pi_a(\psi) \, $,  then  $ \, j_t \circ \varphi = j_t \circ \psi \, $.  Thus if in addition  $ \, A \,{\buildrel {j_t} \over \longrightarrow}\, \widehat{A} \, $  is injective, then  $ \, \varphi = \psi \; $,  i.e.~$ \, t := \pi_a(\varphi) \, $  is enough to determine  $ \varphi \, $.
\end{free text}

\vskip7pt

\begin{notation}
 %  {\sl $ \underline{\text{Notation}} $:}
  In the following, we shall identify any  $ \, \varphi \in P_a(A) \, $  by its corresponding double sequence  $ \, \undt \, $  (as above) and any  $ \, \phi \in U(A) \, $  by  $ \, t := \phi(\ell) \; $:  thus we shall write  $ \undt $  for  $ \varphi $  and  $ t $  for  $ \phi \; $.
                                                              \par
   Given  $ \, A \in \salg_\bk \, $  and  $ \, \undt \in P_a(A) \, $,  for any  $ \, z(a) = \sum_k z_k a^k \in \Z[a] \, $  we shall use notation  $ \, \undt^{z(a)} := \prod_k {\big( t_k^+ \big)}^{z_k} \, $   --- which is just  $ \, \varphi\big( \ell^{z(a)} \big) = \varphi\Big( \prod_k {\big( \ell^{a^k} \big)}^{z_k} \Big) \, $  if  $ \, \undt = \varphi \, $,  with  $ \, \ell^{z(a)} := \prod_k {\big( \ell^{a^k} \big)}^{z_k} \, $.
\end{notation}

\smallskip

\begin{remarks}  \label{remarks-P_a}
 {\it (a)} \,  By (4.5) it is easy to see   --- using notation of  \S \ref{preliminaries}  ---   that for any  $ \, A \in \salg \, $  one has   $ \; P_a(A) \supseteq \big( 1 + \mathfrak{N}(A_0) \big) \supseteq \big( 1 + A_1^{\,2} \big) \; $.
 \vskip3pt
   {\it (b)} \,  It is clear that  $ \, P_a(\mathbb{C}) = \mathbb{C}^* \! = U(\mathbb{C}) \, $,  with  $ \, \mathbb{C}^* := \mathbb{C} \setminus \{0\} \, $.
%
% Indeed, for any  $ \, x \in \mathbb{C} \, $  and  $ \, k \in \N \, $  one has
% that  $ \, {\exp(x)}^{\pm a^k} := \exp\big(\! \pm a^k \log(x)\big) \, $  is a
% well-defined element in  $ \mathbb{C}^* \, $,  so  $ \, \exp(x) \in
% P_a(\mathbb{C}) \, $.  But  $ \, \big\{ \exp(x) \,\big|\, x \in \mathbb{C}
% \,\big\} = \mathbb{C}^* = U(\mathbb{C}) \, $,  hence  $ \, P_a(\mathbb{C})
% = U(\mathbb{C}) \, $.
%                                                          \par
%
   Now assume  $ \, \KK = \mathbb{C} \, $  and  $ \, A = \mathbb{C} [x_1, \dots, x_m, \xi_1, \dots, \xi_n] \, $, where the  $ x_i $'s  and the  $ \xi_j $'s  respectively are even and odd indeterminates.  Letting  $ (\xi_1,\dots,\xi_n) $  be the ideal generated by the  $ \xi_i $'s,  one has  $ \; P_a \big( \mathbb{C} [x_1, \dots, x_m, \xi_1, \dots, \xi_n] \big) \, = \, \mathbb{C}^* + {(\xi_1,\dots,\xi_n)}^2 \; $;  in particular,
 $ \; P_a \big( \mathbb{C} [x_1, \dots, x_m, \xi_1, \dots, \xi_n] \big) = \,
U \big( {\mathbb{C} [x_1, \dots, x_m, \xi_1, \dots, \xi_n]}_0 \big) \; $.
%
%   $$  P_a \big( \mathbb{C} [x_1, \dots, x_m, \xi_1, \dots, \xi_n] \big)  \; = \;\,
% U \big( {\mathbb{C} [x_1, \dots, x_m, \xi_1, \dots, \xi_n]}_0 \big)  $$
% %
%
                                                          \par
   By the same argument, one has also  $ \, P_a(A) = U(A) \, $   --- that is,  $ \pi_a $  is the identity ---   for all those  $ \, A \in \salg_{\mathbb{C}} \, $  such that  $ \, U(A) = \mathbb{C}^* \! + \mathfrak{N}(A_0) \, $,  where  $ \mathfrak{N}(A_0) $  is the nilradical of  $ A_0 \, $.   \hfill  $ \diamondsuit $
\end{remarks}

\smallskip

   Now let  $ \, V = \bigoplus_{\mu \in \fh^*} \! V_\mu \, $  and  $ M $  be as above.  For any  $ \, H \in \fh_{\Z[a]} := \text{\it Span}_{\,\Z[a]} \big(H_1,H_2,H_3\big) \, $,  $ \, A \in \salg \, $  and  $ \, \undt \in P_a(A) \, $,  the formula
 \vskip-4pt
  $$  h^{[a]}_H(\undt)  \, := \,  \undt^H  \, = \,  {\textstyle \sum_{n=0}^{+\infty}} \, {(\,t - \! 1)}^n \, {\textstyle \Big(\! {H \atop n} \!\Big)}   \eqno \text{with \ }  \, t := t_0^+   \qquad \qquad  $$
yields a well defined element of  $ \rGL\big(V_\bk(A)\big) \, $,  whose action is  $ \; h^{[a]}_H(\undt).v \, := \, \undt^{\mu(H)} \, v \; $  for all  $ \, v \in V_\mu \, $,  $ \, \mu \in \fh^* : V_\mu \not= \{0\} \, $ (this makes sense, since  $ \, \mu(H) \in \Z[a] \, $).  In particular, we shall write  $ \, h^{[a]}_i(\undt) := h^{[a]}_{H_i}(\undt) \, $  for  $ \, i = 1, 2, 3 \, $;  thus if  $ \, H = \sum_{i=1}^3 z_i(a) H_i \in \fh_{\Z[a]} := \text{\it Span}_{\,\Z[a]}\big(H_1,H_2,H_3\big) \, $   --- with  $ \, z_i(a) \in \Z[a] \, $  for all  $ i \, $  ---   one has  $ \; h^{[a]}_H(t) \, = \, h^{[a]}_1(\undt) \, h^{[a]}_2(\undt) \, h^{[a]}_3(\undt) \; $  for all  $ \, \undt \in P_a(A) \, $.

\vskip9pt

\begin{definition}  \label{def_mult_one-par-ssgrp_a-type}
 \;  Let  $ \, H \in \fh_{\Z[a]} \, $  as above.  Consider the morphism  $ \, \widehat{h}^{\,[a]}_H : P_a \longrightarrow \rGL\big(V_\bk\big) \, $   given on objects by  $ \, P_a(A) \,{\buildrel {\widehat{h}^{\,[a]}_H(A)} \over {\relbar\joinrel\relbar\joinrel\relbar\joinrel\longrightarrow}}\, \rGL\big(V_\bk(A)\big) \, $,  $ \, \undt \mapsto h^{[a]}_H(\undt) \, $  (its definition on arrows then should be clear).  We define the supergroup functor  $ h^{[a]}_H $  from  $ \salg $  to  $ \grps $  as being the image of  $ \, \widehat{h}^{\,[a]}_H \; $:  in particular it is given on objects by
 $ \; A \, \mapsto \, h^{[a]}_H(A) := \big\{ h^{[a]}_H(\undt) := \undt^H \;\big|\; \undt \in \! P_a(A) \big\} \; $.
\end{definition}

\vskip3pt

\begin{remark}  \label{representability-P_a}
 {\sl Representability of the functors  $ \, h^{[a]}_H \, $}.
                                                          \par
   By construction, the kernel  $ \, \mathcal{K}_a := \text{\it Ker}\,\big(\widehat{h}^{\,[a]}_H\big) \, $  of  $ \, \widehat{h}^{\,[a]}_H $  (notation as above) is the closed subgroup of  $ P_a $  defined by the ideal  $ \, I(\mathcal{K}_a) \, $  given as follows.  If  $ \, \mu \in \fh^* \, $  and  $ \, \mu(H) = \sum_k z_{\mu,k} \, a^k \in \Z[a] \, $  for some  $ \, z_k \in \Z \, $,  then  $ \, I(\mathcal{K}_a) = \Big( \big\{ \prod_k \big( \ell^{a^k} \big)^{z_{\mu,k}} \! - 1 \,\big|\; \mu \in \fh^* : V_\mu \not= \{0\} \,\big\} \Big) \, $.  As the functor  $ P_a $  is represented by  $ \cL_{a,\bk} \, $,  it follows that its quotient  $ \, h^{[a]}_H \cong P_a \big/ \mathcal{K}_a \, $  is represented by the Hopf algebra  $ \, {\cL_{a,\bk}}^{\text{\it co-}I(\mathcal{K}_a)} := \big\{\, f \in \cL_{a,\bk} \,\big|\, \big( \Delta(f) - f \otimes 1 \big) \in \cL_{a,\bk} \otimes I(\mathcal{K}_a) \big\} \, $  of  {\sl right  $ I(\mathcal{K}_a) $--coinvariants\/}  of  $ \cL_{a,\bk} \, $.  In particular,  {\sl the supergroup functor  $ h^{[a]}_H $  is representable, hence it is itself an affine supergroup}.
                                                            \par
   In the following we shall call the  $ h^{[a]}_H $'s  {\sl multiplicative one-parameter (super)subgroups of  $ a $--type}.
\end{remark}

\medskip

\subsection{Construction of Chevalley supergroups}  \label{const-che-sgroup}

\smallskip

   {\ } \quad   In order to define our Chevalley supergroups, we first need the definition of a suitable  {\sl algebraic group\/}  $ \bG_0 $  associated to  $ \, \fg = D(2,1;a) \, $  and  $ \, V $.

\smallskip

   First, for each  $ \, A \! \in \! \salg \, $  consider the subgroup  $ \, G'_0(A) := \Big\langle h_H(A) \, , \, x_\alpha(A) \,\Big|\, H \! \in \! \fh^0_\Z \, , \, \alpha \! \in \! \Delta_0 \Big\rangle $  generated in  $ \rGL\big(V_\bk(A)\big) $  by the one-parameter supersubgroups  $ h_H(A) \, $  and  $ x_\alpha(A) \, $   ---  $ \, H \! \in \fh^0_\Z \, $,  $ \, \alpha \in \Delta_0 \, $.  Overall, this yields a group functor  $ G'_0 $  defined on  $ \salg $,  which clearly factor through  $ \alg = \alg_\Bbbk \, $,  the category of unital commutative  $ \bk $--algebras.  This  $ G'_0 $  is a presheaf  (cf.~\cite{fg1}, Appendix, Definition A.5),  hence we can define the functor  $ \bG'_0 $  as the sheafification of  $ G'_0 \, $:  on  {\sl local\/}  algebras   --- in  $ \alg $  ---   the functor  $ \bG'_0 $  coincides with the functor of points of the  {\sl classical\/}  Chevalley group-scheme associated with the semisimple Lie algebra  $ \fg_0 $   --- isomorphic to  $ \, {\mathfrak{sl}(2)}^{\oplus 3} \, $  ---   and the  $ \fg_0 $--module  $ V \, $: in particular,  $ \bG'_0 $  is representable.  Inside  $ G'_0 \, $,  the  $ h_H $'s  generate the subgroup  $ \; T'(A_0) := \big\langle\, h_H(A) \;\big|\, H \! \in \! \fh^0_\Z \,\big\rangle \, $,  yielding another supergroup functor which also factors through  $ \alg = \alg_\Bbbk \, $;  its sheafification  $ \bT' $  coincides with  $ T' $  itself, and is a maximal torus in  $ \bG'_0 \, $.

\smallskip

   Second, we consider the subgroup
%
% of  $ \rGL\big(V_\bk(A)\big) $
%
 $ \; T(A) \, := \, \Big\langle\, \Big\{ h^{[a]}_H(A) \;\Big|\; H \! \in \! \fh_{\Z[a]} \,\Big\} \, {\textstyle \bigcup} \; T'(A) \;\Big\rangle \; $  of  $ \rGL\big(V_\bk(A)\big) $,
which for various  $ A $  in  $ \salg $  yields another (sub)group functor  $ \; T : \salg \longrightarrow \grps \; $   --- also factoring through  $ \alg \, $:  like above, we can consider also the sheafification  $ \bT $  of  $ T \, $.
                                                                  \par
   For later use, note that  $ \, \bT' \leq \bG'_0 \leq \rGL\big(V_\bk\big) \, $  and  $ \, \bT' \leq \bT \leq \rGL\big(V_\bk\big) \, $,  with  $ \, \bG'_0 \cap \bT = \bT' \, $.

\vskip9pt

\begin{free text}  \label{descr_T^[a]}
 {\bf Description of  $ T^{[a]} \, $.}  Let us spend a few more words in order to describe  $ T $  and  $ \bT \, $:  indeed, we shall show that  $ T $  is representable, so that  $ \, \bT = T \, $.
                                                                  \par
   Let us write  $ \, T^{[a]} := \big\langle\, h^{[a]}_1 , \, h^{[a]}_2 , \, h^{[a]}_3 \,\big\rangle \, $  for the (supersub)group functor generated by  $ \, h^{[a]}_1 $,  $ \, h^{[a]}_2 $  and  $ \, h^{[a]}_3 $,  so that  $ \, T = \big\langle\, T^{[a]} \cup T' \,\big\rangle \, $.  It is clear by construction that  $ \, T^{[a]} $  is a direct product  $ \, T^{[a]} \cong h^{[a]}_1 \times h^{[a]}_2 \times h^{[a]}_3 \, $,  just like  $ \, T' \cong U \times U \times U \, $,  and both these groups commute with each other inside  $ \, \rGL\big(V_\bk\big) \, $;  also, it is clear that the morphism  $ \, \pi_a : P_a \longrightarrow U \, $  uniquely induces a similar morphism  $ \, \pi_a : \mathcal{H}^{[a]} \longrightarrow T' \, $.  It follows that  $ \, T = \big\langle\, T^{[a]} \cup T' \,\big\rangle \, $  can be seen as the fibered product  $ \, T \! \times_{{}_{T'}} \! T' \, $  of  $ \, T^{[a]} $  and  $ T' $  with respect to the pair of morphisms  $ \, T^{[a]} \,{\buildrel {\pi_a} \over {\relbar\joinrel\longrightarrow}}\; T' \,{\buildrel {\;\;\text{\sl id}_{T'}} \over {\longleftarrow\joinrel\relbar}}\, T' \, $.  We shall now realize this fibered product in concrete terms.

\smallskip

   Take on  $ \, T^{[a]} \! \times T' \, $  the direct product structure.  Recall that  $ T $  is representable, hence it coincides with its sheafification  $ \bT $;  similarly we see by construction that  $ T^{[a]} $  is representable too, so for its sheafification  $ \bT^{[a]} $  we have  $ \, \bT^{[a]} = T^{[a]} \, $.  It follows that  $ \, T^{[a]} \times T' = \bT^{[a]} \times \bT' \, $  is representable too. The fibered product  $ \, T^{[a]} \! \times_{{}_{T'}} \! T' \, $  is a quotient of the above direct product.  Indeed, let
 $ \, K(A_0) := \big\{ \big( x, y \big) \,\big|\, x \in T^{[a]}(A_0) , \, y \in T'(A_0) , \, \pi_a(x) = y^{-1} \big\} \, $
for any  $ \, A \in \salg \, $,  so that  $ \; A \mapsto K(A_0) \; $  defines   --- on  $ \salg $,  through  $ \text{(alg)} $  ---   a (normal) subgroup functor  $ K $  of  $ \; T^{[a]} \! \times T' \; $:  then we have a (functor) isomorphism
 $ \; T \, \cong \, \big( T^{[a]} \! \times T' \big) \Big/ K \; $.
 In addition, let  $ \mathbf{K} $  denote the sheafification of the functor  $ K \, $.

\smallskip

   We see now that  $ \mathbf{K} $  as a subgroup of  $ \, T^{[a]} \! \times T' \, $  is  {\sl closed}.  To begin with, recall that   --- by construction ---   the group multiplication provides isomorphisms  $ \, h^{[a]}_1 \! \times h^{[a]}_2 \! \times h^{[a]}_3 \cong T^{[a]} \, $  and  $ \, h_{H_{2\varepsilon_1}} \!\! \times h_{H_{2\varepsilon_2}} \!\! \times h_{H_{2\varepsilon_3}} \cong T' \, $:  \, thus we can write %%
%
% the group of  $ A $--points  (for  $ \, A \in \salg $)  of  $ \, T^{[a]} \! \times T' \, $  as
%   $$  \big(T^{[a]} \! \times T'\big)(A)  \,\; = \;\,  \big\{ \big( (\, \undt_1 , \undt_2 , \undt_3) \, , % (\tau_1 , \tau_2 , \tau_3) \big) \,\big|\; \undt_i \in P_a(A) \, , \, \tau_i \in U(A) \, , \; \forall
% \; i \in \{1,2,3\} \big\}  $$
% %
%%
%
 the  $ A $--points  (for  $ \, A \in \salg $)  of  $ \, T^{[a]} \! \times T' \, $  as pairs of triples of the form  $ \; \big( (\, \undt_1 , \undt_2 , \undt_3) \, , (\tau_1 , \tau_2 , \tau_3) \big) \; $  with  $ \, \undt_1, \undt_2, \undt_3 \in P_a(A) \, $  and  $ \, \tau_1, \tau_2, \tau_3 \in U(A) \, $   --- with a slight abuse of notation: we are identifying  $ \, h^{[a]}_i(\undt_i) \, $  with  $ \, \undt_i \, $  and  $ \, h_{H_{2\varepsilon_j}}\!(\tau_j) \, $  with  $ \, \tau_j \, $,  for all  $ i $  and  $ j \, $.
                                                             \par
   Recall the identities
%%
%   $$  H_{2\varepsilon_1} \, = \, H_2 \;\; ,  \qquad  H_{2\varepsilon_3} \, = \, H_3 \;\; ,
% \qquad  H_{2\varepsilon_2} \, = \, 2 H_1 - (1+a) H_2 - a H_3  $$
% %
%%
 $ \; H_{2\varepsilon_1} = H_2 \; $,  $ \; H_{2\varepsilon_3} = H_3 \; $,  $ \; H_{2\varepsilon_2} = 2 \, H_1 - (1+a) H_2 - a H_3 \; $
which hold inside  $ \fh \, $;  in turn, these yield, for every  $ \, \undt \in P_a(A) \, $  with  $ \, \vartheta := \pi_a(\undt) \in U(A) \, $,  formal identities
  $$  \displaylines{
   h_{H_{2\varepsilon_1}}\!(\vartheta) \, = \, \vartheta^{H_{2\varepsilon_1}} = \, \undt^{H_2} = \, h_{H_2}^{[a]}(\undt) \, = \, h_2^{[a]}(\undt) \;\; ,  \qquad  h_{H_{2\varepsilon_3}}\!(\vartheta) \, = \, \vartheta^{H_{2\varepsilon_3}} = \, \undt^{H_3} = \, h_{H_3}^{[a]}(\undt) \, = \, h_3^{[a]}(\undt)  \cr
   h_{H_{2\varepsilon_2}}\!(\vartheta)  \; = \;  \vartheta^{H_{2\varepsilon_2}}  \, = \;  \undt^{2 H_1 - (1+a) H_2 - a H_3}  \, = \;  h_{2 H_1}^{[a]}\!(\undt) \cdot h_{-(1+a)H_2}^{[a]}\!(\undt) \cdot h_{-a H_3}^{[a]}\!(\undt)  }  $$
which in shorter notation read
  $$  (\vartheta,1,1) \, = \, \pi_a(1,\undt\,,1)  \;\; ,  \;\quad  (1,1,\vartheta) \, = \, \pi_a(1,1,\undt\,)  \;\; ,  \;\quad  (1,\vartheta,1) \, = \, \pi_a\big(\, \undt^2, \undt^{-(1+a)}, \undt^{-a} \big)   \qquad   \eqno (4.6)  $$
   \indent   Now consider the condition  $ \, \pi_a(x) = y^{-1} \, $  for a pair  $ (x,y) $  to belong to  $ K(A) $,  rewritten as  $ \; y \, \pi_a(x) = 1 \; $:  we want to read it for any pair  $ \; (x,y) := \big( (\, \undt_1 , \undt_2 , \undt_3) \, , (\tau_1 , \tau_2 , \tau_3) \big) \; $  as above.  By the previous analysis, it corresponds to the three equations which in turn correspond to the three conditions in (4.6).
 Recall that the group  $ \, T^{[a]} \! \times T' \, $  is represented by the algebra  $ \, \cO \big( T^{[a]} \! \times T' \big) \cong \cO\big( T^{[a]} \big) \otimes \cO\big(T'\big) \; $:  the left-hand tensor factor is a quotient of
 $ \; \cO\big( P_a^{\times 3}\big) \, \cong \, {\cO(P_a)}^{\otimes 3} \cong \,
 \allowbreak
 \bigotimes\limits_{i=1}^3 \bk\big[ \big\{\ell_i^{\,\pm a^k}\big\}_{k \in \N} \,\big] \; $,
while the right-hand one is
 $ \; \cO\big(T'\big) \, \cong \, \cO\big(U^{\times 3}\big) \, \cong \, \bigotimes\limits_{i=1}^3 \bk\big[ z_i^{\pm 1} \big] \; $  (cf.~\S \ref{aff_sgrps-P_a}).
 Then first two conditions in (4.6) correspond to the equations
 $ \; z_1 \, \ell_2 = 1 \; $,  $ \; z_3 \, \ell_3 = 1 \; $
in  $ \, \cO\big( T^{[a]} \! \times T' \big) \, $.
                                                                      \par
   As to the third condition, we can handle it as follows.  For any  $ \, \undt \in P_a(A) \, $,  let  $ \, \widehat{\undt} := \big( \undt_1, \undt_2, \undt_3 \big) \in {P_a(A)}^{\times 3} \, $  with  $ \, \undt_1 =: p_1\big(\,\widehat{\undt}\,\big) = \undt \; $,  where  $ p_1 $  is the projection of  $ \, {P_a(A)}^{\times 3} $  onto its leftmost factor, and consider the (functorial) ``diagonalisation map''  $ \, P_a \,{\buildrel {\varDelta_3} \over {\relbar\joinrel\longrightarrow}}\, {P_a(A)}^{\times 3} \, $,  $ \, \undt \mapsto \big(\undt,\undt,\undt\big) \, $;  then  $ \, \varDelta_3 \circ p_1 \, $  is a (group functor) morphism, and
 $ \, \pi_a\big(\, \undt^2, \undt^{-(1+a)}, \undt^{-a} \big) = \big(\, \widehat{\ell}_1^{\;2}, \widehat{\ell}_2^{\;-(1+a)}, \widehat{\ell}_3^{\;-a} \big)\big(\,\widehat{\undt}\,\big) \, $  where  $ \, \widehat{\ell}_i := \ell_i \circ \varDelta_3 \circ p_1 \, \in \, \cO\big(P_a^{\times 3}\big) \, \cong \, \bk\big[ \big\{ \ell_1^{\,\pm a^k}, \, \ell_2^{\,\pm a^k}, \, \ell_3^{\,\pm a^k} \big\}_{k \in \N} \,\big] \, $.
 The outcome is that the third condition in (4.6) corresponds to the equation  $ \; z_2 \, \widehat{\ell}_1^{\;2} \, \widehat{\ell}_2^{\;-1} \, \widehat{\ell}_2^{\;-a} \, \widehat{\ell}_3^{\;-a} \, = \, 1 \; $.  To sum up, we get that  $ K $  is the closed subgroup of  $ \, T^{[a]} \! \times T' \, $  defined by the ideal  $ \; I(K) = \big( z_1 \, \ell_2 - 1 \, , \, z_3 \, \ell_3 - 1 \, , \, z_2 \, \widehat{\ell}_1^{\;2} \, \widehat{\ell}_2^{\;-1} \, \widehat{\ell}_2^{\;-a} \, \widehat{\ell}_3^{\;-a} - 1 \big) \; $.

\vskip5pt

   Finally, as  $ \, T^{[a]} \! \times T' \, $  is representable (hence it is an affine group scheme) and its subgroup  $ K $  is closed, we argue that  $ \, \mathbf{K} = K \, $,  that the latter is also representable and hence the quotient  $ \; T \, \cong \, \big( T^{[a]} \! \times T' \big) \Big/ K \; $  is representable too: in particular,  $ \, \bT = T \, $  as well.
\end{free text}

\vskip7pt

   We can now introduce the algebraic group  $ \bG_0 $  we were looking for:

\begin{definition}
 \; For every  $ \, A \in \salg \, $,  we let  $ G_0(A) $  be the subgroup
 $ \; G_0(A) := \Big\langle G'_0(A) \, {\textstyle \bigcup} \, T(A) \Big\rangle \; $
of  $ \rGL\big(V_\bk(A)\big) \; $.  We denote  $ \; G_0 : \salg \longrightarrow \grps \; $  the supergroup functor which is the full subfunctor of  $ \rGL\big(V_\bk(A)\big) $  given on objects by  $ \, A \mapsto G_0(A) \, $;  we denote  $ \; \bG_0 : \salg \longrightarrow \grps \; $  the sheafification functor of $ G_0 \, $.
%                                                                           \par
   {\sl Note\/}  that both  $ G_0 $  and  $ \bG_0 $  factor through  $ \alg \, $.
\end{definition}

\smallskip

\begin{proposition}  \label{bG_0-repres}
 \; The supergroup functor  $ \bG_0 $  is representable, hence   --- as it factors through  $ \alg $  ---   it is an affine group.
\end{proposition}

\begin{proof}
 We argue much like we did above  (cf.~\S \ref{descr_T^[a]})  to describe  $ T $,  so we can be more sketchy.

\smallskip

   The groups  $ G'_0 $  and  $ T $  are subgroups of  $ \rGL(V_\bk) \, $,  and their mutual intersection is  $ \, G'_0 \cap T = T' \, $:  thus  $ G_0 $  is a fibered product of  $ G'_0 $  and  $ T $  over  $ T' \, $,  that we can describe it in down-to-earth terms.

\smallskip

   Fix  $ \, A \in \salg \, $.  Inside  $ \rGL(V_\bk)(A_0) \, $,  the subgroup  $ T(A_0) $  acts on  $ G'_0(A_0) $  by adjoint action, so  $ G_0(A_0) \, $,  generated by  $ T(A_0) $  and  $ G'_0(A_0) \, $,  is a quotient of the semi-direct product  $ \, T(A_0) \ltimes G'_0(A_0) \, $.  Indeed, let
 $ \, J(A_0) := \big\{ \big( j^{-1}, j \big) \,\big|\, j \in T(A_0) \cap G'_0(A_0) = T'(A_0) \big\} \, $,  so that  $ \; A \mapsto J(A_0) \; $  defines   --- on  $ \salg $,  through  $ \text{(alg)} $  ---   a normal subgroup functor of  $ \; T \ltimes G'_0 \; $:  then we have a (functor) isomorphism
 $ \; G_0 \, \cong \, \big( T \ltimes G'_0\big) \Big/ J \; $.
 Thus  $ \; G_0 \cong \big( T \ltimes G'_0 \big) \Big/ J \; $  as group functors, hence (forgetting the group structure) also  $ \, G_0 \cong \big( T \! \times G'_0 \big) \Big/ \! J \, $  as set-valued functors.  Taking sheafifications, as both  $ \bT $  and  $ \bG'_0 $  are representable, we infer that  $ \, \bT \times \bG'_0 \, $  is representable too.  In addition, if  $ \mathbf{J} $  is the sheafification of the functor  $ J $,  we see that  $ \mathbf{J} $  as a subgroup of  $ \, \bT \times \bG'_0 \, $  is  {\sl closed}.
                                                     \par
   To see all this in detail, let us revisit our construction.  We started with a representation of  $ \, A \otimes_\Bbbk U_\Bbbk(\fg) \, $  on  $ \, V_\Bbbk(A) := A \otimes_\Bbbk V_\Bbbk \, $  (cf.~\S \ref{adm-lat}).  By construction,  $ \bG'_0 $  is just the algebraic group associated to  $ \fg_0 $  and  $ V $  (as a  $ \fg_0 $--module)  by the classical Chevalley's construction: indeed,  $ \, \fg_0 \cong \rsl_2 \oplus \rsl_2 \oplus \rsl_2 \, $  where the three summands are given by  $ \rsl_2 $--triples  associated to positive even roots  $ \, 2\,\varepsilon_i \, $  (cf.~\S \ref{def_D(2,1;a)}),  and  $ \, \bG'_0 \cong \mathbf{H}_1 \times \mathbf{H}_2 \times \mathbf{H}_3 \, $  where  $ \, \mathbf{H}_i \in \big\{ \mathbf{SL}_2 \, , \mathbb{P}\mathbf{SL}_2 \big\} \, $  for each  $ i \, $.  Each  $ \mathbf{H}_i $  is represented by  $ \, \cO(\mathbf{SL}_2) = \bk\big[ \text{a}, \text{b}, \text{c}, \text{d} \big] \Big/ \big( \text{a} \text{d} - \text{b} \text{c} - 1 \big) \, $  if  $ \, \mathbf{H}_i \cong \mathbf{SL}_2 \, $,  and by the unital subalgebra of  $ \cO(\mathbf{SL}_2) $  generated by all products of any two elements in the set  $ \, \big\{ \text{a}, \text{b}, \text{c}, \text{d} \big\} \, $  if  $ \, \mathbf{H}_i \cong \mathbb{P}\mathbf{SL}_2 \, $.  Then the torus  $ \, T = \bT \cong U^{\times 3} \, $  is embedded in  $ \, \times_{i=1}^3 \mathbf{H}_i \, $  via  $ \, (\tau_1,\tau_2,\tau_3) \mapsto \Big(\! \Big(\!\!\! {{\tau_1 \;\;\ 0} \atop {\;\; 0 \,\;\ \tau_1^{-1}}} \!\Big) , \Big(\!\!\! {{\tau_2 \;\;\ 0} \atop {\;\; 0 \,\;\ \tau_2^{-1}}} \!\Big) , \Big(\!\!\! {{\tau_3 \;\;\ 0} \atop {\;\; 0 \,\;\ \tau_3^{-1}}} \!\Big) \!\Big) \, $,  hence it is the closed subgroup of  $ \bG'_0 $  defined by the ideal  $ \, I(T) = \big( \text{b}_1 \, , \text{b}_2 \, , \text{b}_3 \, , \text{c}_1 \, , \text{c}_2 \, , \text{c}_3 \big) \, $.
                                                            \par
   Now the description of the subgroup  $ J $  of  $ \, \bT \times \bG' \, $  goes much along the same lines as for describing the subgroup  $ \bT' $  of  $ \, T^{[a]} \! \times T' \, $  in  \S \ref{aff_sgrps-P_a}  above.  Then by the same arguments one eventually finds that  $ \, J = \mathbf{J} \, $  is a closed in  $ \, \bT \times \bG'_0 \, $, as claimed.
                                                            \par
   A similar analysis works too when some  $ \mathbf{H}_i $'s  (possibly all of them) are isomorphic to  $ \mathbb{P}\mathbf{SL}_2 \, $.

\vskip5pt

   Finally, as  $ \, \bT \times \bG'_0 \, $  is representable (so it is an affine group scheme), its quotient by the closed normal subgroup  $ \mathbf{J} $  is representable too, hence the same holds for the isomorphic functor  $ \bG_0 \, $.
\end{proof}

\vskip5pt

   We can now eventually define our Chevalley supergroups:

\vskip13pt

\begin{definition} \label{def_Che-sgroup_funct} {\ }
 Let  $ \fg $  and  $ V $  be as above.  We call  {\it Chevalley supergroup functor},  associated to  $ \fg $  and  $ V \, $,  the functor  $ \; G : \salg \lra \text{(grps)} \; $  given by:
 \vskip3pt
   \noindent \quad   --- if  $ \, A \! \in \! \text{\it Ob}\big(\salg\big) \, $  we let  $ \, G(A) \, $  be the subgroup of  $ \rGL\big(V_\bk(A)\big) $  generated by  $ G_0(A) $  and the one-parameter subgroups  $ x_\beta(A) $  with  $ \, \beta \in \Delta_1 \, $,  that is
 \vskip-4pt
  $$  G(A)  \; := \;  \Big\langle\, G_0(A) \, , \, x_\beta(A) \,\;\Big|\;\, \beta \in \Delta_1 \, \Big\rangle  \; = \;  \Big\langle T(A) \, , \, x_\delta(A) \;\Big|\; \delta \in \Delta \Big\rangle  $$
 \vskip-2pt
\noindent
 where the second identity follows from the previous description of  $ G_0 \, $.
 \vskip2pt
   --- if  $ \, \phi \in \text{Hom}_{\salg}\big(A\,,B\big) $,  then  $ \; \End_\bk(\phi) : \End_\bk\big(V_\bk(A)\big) \! \lra \End_\bk\big(V_\bk(B)\big) \,$  (given on matrix entries by  $ \phi $  itself)  respects the sum and the associative product of matrices; then  $ \End_\bk( \phi) $  clearly restricts to a group morphism  $ \, \rGL\big(V_\bk(A)\big) \lra \rGL\big(V_\bk(B)\big) \, $.
 The latter maps the generators of  $ G(A) $  to those of  $ G(B) $,  hence restricts to a group morphism  $ \; G(\phi) : G(A) \lra G(B) \; $.
 \vskip3pt
   We call {\it Chevalley supergroup}   --- associated to  $ \, \fg = D(2,1;a) \, $  and  $ V $  ---   the sheafification  $ \bG $  of $ G \, $  (cf.~\cite{fg1}, Appendix).  Thus  $ \; \bG : \salg \lra \text{(grps)} \; $  is a sheaf functor such that  $ \, \bG(A) = G(A) \, $  if  $ \, A \in \salg \, $  is  {\sl local}.
%
%                                                                    \par
%    To stress the dependence on  $ V $,  we shall sometimes write  $ G_V $  and  $ \bG_V $  for  $ G $
% and  $ \bG $  respectively.
%
 To stress the dependence on  $ V $,  we shall also write  $ G_V $  for  $ G $  and  $ \bG_V $  for  $ \bG \, $.
\end{definition}

\vskip3pt

\begin{remarks}   \label{rem_alt-constr}
 {\it (a)} \,  For  $ \, a \in \Z \, $,  a construction of Chevalley supergroups of type  $ D(2,1;a) $  was given in \cite{fg1}:  it coincides with the present one, because  $ \, P_a(A) = U(A_0) \, $  if  $ \, a \in \Z \, $   --- for  $ \, A \in \salg \, $.
 \vskip3pt
   {\it (b)} \,  An alternative definition of Chevalley supergroups can be given letting the subgroup (functor)  $ \, \text{\it Up} : A \mapsto \text{\it Up}(A_0) \, $  play the role of  $ P_a \, $,  where  $ \text{\it Up}(A_0) := \big( 1 \! + \mathfrak{N}(A_0) \big) \, $  is the subgroup of  $ U(A_0) $  of all  {\sl unipotent\/}  elements of  $ A_0 \, $  (and  $ \mathfrak{N}(A_0) $  is the nilradical of  $ A_0 $).  All our arguments and results from now on still stand valid as well.  Nevertheless, using the subgroup functor  $ \text{\it Up} \, $  one does not recover the construction of  \cite{fg1}  for  $ \, a \in \Z \, $,  which instead is the case with  $ P_a \, $,  see  {\it (a)\/}  above.
\end{remarks}

\medskip

\subsection{Chevalley supergroups as affine supergroups}
\label{che-sgroup_alg}

\smallskip

   {\ } \quad   Our definition of the Chevalley supergroup  $ \bG $  does not imply (at first sight) that  $ \bG $  is  {\it representable},  so that it is indeed an affine supergroup scheme.  In this section we prove this fact.

\medskip

\begin{definition}  \label{subgrps}
 For any  $ \, A \in \salg \, $,  we define the subsets of  $ G(A) $
 \vskip-4pt
  $$  \begin{array}{rl}
   G_1(A)  &  \! := \;  \left\{\, {\textstyle \prod_{\,i=1}^{\,n}} \,
x_{\gamma_i}(\vartheta_i) \;\Big|\; n \in \N \, , \; \gamma_i \in \Delta_1 \, ,
\; \vartheta_i \in A_1 \,\right\}_{\phantom{\Big|}}  \\
   G_0^\pm(A)  &  \! := \;  \left\{\, {\textstyle \prod_{\,i=1}^{\,n}} \,
x_{\alpha_i}(t_i) \;\Big|\; n \in \N \, , \; \alpha_i \in \Delta^\pm_0 \, ,
\; t_i \in A_0 \,\right\}_{\phantom{\Big|}}  \\
   G_1^\pm(A)  &  \! := \;  \left\{\, {\textstyle \prod_{\,i=1}^{\,n}} \,
x_{\gamma_i}(\vartheta_i) \;\Big|\; n \in \N \, , \; \gamma_i \in \Delta^\pm_1 \, ,
\; \vartheta_i \in A_1 \,\right\}_{\phantom{\Big|}}  \\
   G^\pm(A)  &  \! := \;  \left\{ {\textstyle \prod_{\,i=1}^{\,n}} \,
x_{\beta_i}(\bt_i) \;\Big|\; n \! \in \! \N \, , \, \beta_i \! \in \! \Delta^\pm ,
\, \bt_i \! \in \! A_0 \! \cup \! A_1 \right\}  \, = \;
\big\langle G_0^\pm(A) \, , G_1^\pm(A) \big\rangle
      \end{array}  $$
Moreover, fixing any total order  $ \, \preceq \, $  on $ \Delta_1^\pm \, $,  and letting  $ \, N_\pm = \big| \Delta_1^\pm \big| \, $,  we set
  $$  G_1^{\pm,<}(A)  \,\; := \;  \left\{\; {\textstyle \prod_{\,i=1}^{\,N_\pm}} \, x_{\gamma_i}(\vartheta_i) \,\;\Big|\;\, \gamma_1 \prec \cdots \prec \gamma_{N_\pm} \in \Delta^\pm_1 \, , \; \vartheta_1, \dots, \vartheta_{N_\pm} \in A_1 \,\right\}  $$
and for any total order  $ \, \preceq \, $  on  $ \Delta_1 \, $,  and letting  $ \, N := \big| \Delta \big| = N_+ + N_- \, $,  we set
  $$  G_1^<(A)  \,\; := \;  \left\{\, {\textstyle \prod_{\,i=1}^N} \, x_{\gamma_i}(\vartheta_i) \,\;\Big|\;\, \gamma_1 \prec \cdots \prec \gamma_{\scriptscriptstyle N} \in \Delta_1 \, , \; \vartheta_1, \dots, \vartheta_N \in A_1 \,\right\}  $$
By the way, note that  $ \, N_\pm = 4 \, $  and  $ \, N = 8 \, $  for  $ \, \fg = D(2,1;a) \, $.
                                                                     \par
   Similar notations will denote the sheafifications  $ \bG_1 \, $,
$ \bG^\pm $,  $ \bG_0^\pm $,  $ \bG_1^\pm $,  etc.
\end{definition}

\vskip7pt

   We begin studying commutation relations among generators of Chevalley groups.  As a matter of notation, when  $ \varGamma $  is any group and  $ \, g, h \in \varGamma \, $  we denote by  $ \, (g,h) := g \, h \, g^{-1} \, h^{-1} \, $  their commutator.

\vskip13pt

\begin{lemma}  \label{comm_1-pssg} {\ }
\vskip9pt
\noindent
 \;\; (a) \,  Let  $ \; \alpha \in \Delta_0 \, $,  $ \, \gamma \in \Delta_1 \, $,  $ \, A \in \salg \, $  and  $ \, t \in A_0 \, $,  $ \, \vartheta \in A_1 \, $.  Then there exist  $ \, c_s \! \in \! \Z \, $  such that
  $$  \big( x_\gamma(\vartheta) \, , \, x_\alpha(t) \big)  \, = \,  {\textstyle \prod_{s>0}} \, x_{\gamma + s \, \alpha}\big(c_s \, t^s \vartheta\big)  \;\; \in \;\;  G_1(A)  $$
(the product being finite).  Indeed, with  $ \, \varepsilon_k = \pm 1 \, $  and  $ \, r \, $  as in \S \ref{comm-rul_kost},  part (2),
  $$  \big( 1 + \vartheta \, X_\gamma \, , \, x_\alpha(t) \big)  \, = \,
{\textstyle \prod_{s>0}} \, \Big( 1 + {\textstyle \prod_{k=1}^s \varepsilon_k \cdot {{s+r} \choose r}} \cdot t^s \vartheta \, X_{\gamma + s \, \alpha} \Big)  $$
where the factors in the product are taken in any order (as they do commute).
\vskip9pt
\noindent
 \;\; (b)  Let  $ \; \gamma , \delta \! \in \! \Delta_1 \, $,  $ \, A \! \in \! \salg \, $,  $ \, \vartheta, \eta \! \in \! A_1 \, $.  Then (notation of  Definition \ref{def_che-bas})
  $$  \big( x_\gamma(\vartheta) \, , \, x_\delta(\eta) \big)  \; = \;\,  x_{\gamma + \delta}\big(\! - \! c_{\gamma,\delta} \; \vartheta \, \eta \big)  \; = \;  \big(\, 1 \! - \! c_{\gamma,\delta} \; \vartheta \, \eta \, X_{\gamma + \delta} \,\big)  \;\; \in \;\;  G_0(A)  $$
if  $ \; \delta \not= -\gamma \; $;  otherwise, for  $ \; \delta = -\gamma \, $,  we have
  $$  \big( x_\gamma(\vartheta) \, , \, x_{-\gamma}(\eta) \big)  \; = \;  \big(\, 1 \! - \vartheta \, \eta \, H_\gamma \,\big)  \; = \;  h^{[a]}_{H_\gamma}\!\big( 1 \! - \vartheta \, \eta \big) \;\; \in \;\; G_0(A)  $$
 \vskip3pt
\noindent
 \;\; (c)  Let  $ \; \beta \in \Delta \, $,  $ \, A \in \salg \, $,  $ \, \bu \in A_0 \! \cup \! A_1 \, $.  Then
  $$  \displaylines{
   \hskip7pt   h_{H_0}(t) \; x_\beta(\bu) \; {h_{H_0}(t)}^{-1}  = \;  x_\beta\big( t^{\beta(H_0)} \, \bu \big)  \; \in \;  G_{\!p(\beta)}\!(A)   \hskip15pt  \forall \;\; H_0 \in \fh^0_\Z \, , \; t \in U(A_0)  \cr
   \hskip7pt   h^{[a]}_H(\undt) \; x_\beta(\bu) \; {h^{[a]}_H(\undt)}^{-1}  = \;  x_\beta\big( \undt^{\beta(H)} \, \bu \big)  \, \in \,  G_{\!p(\beta)}\!(A)   \hskip13pt  \forall \;\; H \in \fh_{\Z[a]} \, , \; \undt \in P_a(A_0)  }  $$
where  $ \; p(\beta) := s \, $,  by definition, if and only if  $ \, \beta \in \Delta_s \; $.
\end{lemma}

\begin{proof} Like for  \cite{fg1},  Lemma 5.13, these results follow directly from the classical ones
%
% in  \cite{st},  pg.~22 and 29,
%
 and simple calculations, using the relations in  \S \ref{comm-rul_kost}  and the identity  $ \, {(\vartheta \eta)}^2 = - \vartheta^2 \, \eta^2 = 0 \, $.  In addition, for  {\it (b)\/}  in the present case we have also to take into account that  $ \; (1 - \vartheta \, \eta) \, \in \, \big( 1 + \mathfrak{N}(A_0) \big) \, \subseteq \, P_a(A) \; $  for all  $ \, \vartheta , \eta \in A_1 \, $,  so  $ \; h^{[a]}_{H_\gamma} \! \big( 1 \! - \vartheta \, \eta \big) \; $  is well-defined and equal to  $ \; \big(\, 1 \! - \vartheta \, \eta \, H_\gamma \,\big) \; $.
\end{proof}

\vskip11pt

   Now, with our definition of Chevalley groups at hand and the commutation rules among their generators available   -- as given in  Lemma \ref{comm_1-pssg}  above ---   one can reproduce whatever was done in  \cite{fg1},  \S 5.3.
%
% The results then will be the same, and, for their proofs, exactly the same arguments will work.
%
 Just some minimal changes are in order, due to a couple of facts: first, the presence among generators of the multiplicative one-parameter supersubgroups of type  $ a $,  which are handled like the classical ones using  Lemma \ref{comm_1-pssg};  second, several shortcuts and simplifications are possible, as the structure of  $ \, \fg = D(2,1;a) \, $  is much simpler than that of the general classical Lie superalgebra.  Thus in the following I shall bound myself to list the results we get (essentially, the main steps in the line of arguing of  \cite{fg1}),  as the proofs can be easily recovered from  \cite{fg1}, \S 5.  The first result is:

\vskip11pt

\begin{theorem} \label{g0g1}
 For any  $ \, A \in \salg \, $,  there exist set-theoretic factorizations
 \vskip-9pt
  $$  \begin{matrix}
   \text{\it (a)}  &  \qquad  G(A)  \; = \;  G_0(A) \; G_1(A) \quad ,  \phantom{\Big|}  &
\;\quad  G(A)  \; = \;  G_1(A) \; G_0(A)  \\
   \text{\it (b)}  &  \qquad  G^\pm(A)  \; = \;  G^\pm_0(A) \; G^\pm_1(A) \quad ,  \phantom{\Big|}  &
\;\quad  G^\pm(A)  \; = \;  G^\pm_1(A) \; G^\pm_0(A)  \\
   \text{\it (c)}  &  \qquad  G(A)  \; = \;  G_0(A) \, G_1^<(A) \quad ,  \phantom{\Big|}  &
\;\quad  G(A)  \; = \;  G_1^<(A) \, G_0(A)
      \end{matrix}  $$
\end{theorem}

\smallskip

   Claim  {\it (a)\/}  above is a group-theoretical counterpart for  $ G $  of the splitting  $ \, \fg = \fg_0 \oplus \fg_1 \, $   --- a super-analogue of the Cartan decomposition for reductive groups ---   while  {\it (b)\/}  is a similar result for  $ G^+ $  and  $ G^- \, $.  Claim  {\it (c)\/}  improves  {\it (a)\/}:  like in  \cite{fg1}  (Theorem 5.17), it follows from  Theorem \ref{g0g1}  just reordering the factors in  $ G_1(A) $  by means of the commutation rules in  Lemma \ref{comm_1-pssg}.

\medskip

   With a more careful analysis, one improves the previous results to get the following, key one:

\medskip

\begin{theorem} \label{nat-transf}
   The group product yields isomorphisms of set valued functors
 \vskip-15pt
  $$  \displaylines{
   G_0 \times G_1^{-,<} \times G_1^{+,<} \;{\buildrel \cong \over {\relbar\joinrel\lra}}\; G \quad ,  \qquad  \bG_0 \times \bG_1^{-,<} \times \bG_1^{+,<} \;{\buildrel \cong \over {\relbar\joinrel\lra}}\; \bG  }  $$
 \vskip-3pt
\noindent
 as well as those obtained by permuting the  $ (-) $-factor  and the  $ (+) $-factor  and/or moving the  $ (0) $-factor  to the right.  All these induce similar functor isomorphisms with the left-hand side obtained
by permuting the factors above, like  $ \, G_1^{+,<} \times G_0 \times G_1^{-,<} \;{\buildrel \cong \over \lra}\; G \, $,  $ \, \bG_1^{-,<} \times \bG_0 \times \bG_1^{+,<} \;{\buildrel \cong \over \lra}\; G \, $,  etc.
\end{theorem}

\smallskip

   The same technique used to prove  Theorem \ref{nat-transf}  also yields the following:

\smallskip

\begin{proposition} \label{G1pm-repres}
  The functors  $ \, G_1^{\pm,<} : \salg \lra \sets \; $  are representable: they are the functor of points of the superscheme  $ \, \mathbb{A}_\bk^{0|N_\pm} $,  with  $ \, N_\pm = \big| \Delta_1^\pm \big| \, $.  In particular  $ \; G_1^{\pm,<} = \bG_1^{\pm,<} \; $.
\end{proposition}

\smallskip

   Indeed, this holds since natural transformations  $ \, \Psi^\pm : \mathbb{A}_\bk^{0|N_\pm} \! \lra G_1^{\pm,<} \; $  exist, by definition, given on objects by  $ \; \Psi^\pm(A) : \, \mathbb{A}_\bk^{0|N_\pm}(A) \! \lra G_1^{\pm,<}(A) \, $,  $ \, (\vartheta_1,\dots,\vartheta_{N_\pm}) \, \mapsto \, {\textstyle \prod_{i=1}^{N_\pm}} \, x_{\gamma_i}(\vartheta_i) \, $   --- and obvious on morphisms.  One then proves that these  $ \Psi^\pm $  are isomorphisms of (set valued) functors.

\smallskip

   Finally, we can prove that the Chevalley supergroups are affine:

\medskip

\begin{theorem}  \label{representability}
   Every Chevalley supergroup  $ \bG $  is an affine supergroup.
\end{theorem}

\smallskip

   Indeed, this is a direct consequence of the last two results, as they imply that the group functor  $ \bG $  is isomorphic (as a set valued functor) to the direct product of three representable group functors, hence in turn it is representable as well, which entails that it is an affine supergroup.

\smallskip

   Another immediate consequence is the following, which improves, for Chevalley supergroups, a more general result proved by Masuoka  (cf.~\cite{ms},  Theorem 4.5) in the algebraic setting.
%
% (see also  \cite{vis}, and references therein,  for the complex-analytic case).
%

\smallskip

\begin{proposition}
 For any Chevalley supergroup  $ \bG \, $,  there are isomorphisms
%
% %
%  \vskip-11pt
% %
%   $$  \begin{array}{rl}
%    \cO(\bG) \hskip-3pt  &  \cong \;  \cO(\bG_0) \otimes \cO\big(\bG_1^{-,<}\big)
% \otimes \cO\big(\bG_1^{+,<}\big)  \; \cong  \phantom{\Big|}  \\
%                  &  \qquad \qquad  \cong \;  \cO(\bG_0) \otimes \bk \big[ \xi_1, \dots, \xi_{N_-} \big]
% \otimes \bk \big[ \chi_1, \dots, \chi_{N+} \big]
%       \end{array}  $$
% %
%
 \vskip-17pt
   $$  \cO(\bG)  \,\; \cong \;\,  \cO(\bG_0) \otimes \cO\big(\bG_1^{-,<}\big) \otimes \cO\big(\bG_1^{+,<}\big)  \,\; \cong \;\,  \cO(\bG_0) \otimes \bk \big[ \xi_1, \dots, \xi_{N_-} \big] \otimes \bk \big[ \chi_1, \dots, \chi_{N+} \big]  $$
 \vskip-5pt
\noindent
 of commutative superalgebras,  where  $ \, N_\pm = \big| \Delta_1^\pm \big| \, $,  the subalgebra
$ \cO(\bG_0) $  is totally even, and  $ \, \xi_1, \dots, \xi_{N_-}
\, $  and  $ \, \chi_1, \dots, \chi_{N_+} \, $  are odd elements.
\end{proposition}

\vskip5pt

   Finally, one has the following result for any  $ \, A  \in \salg \, $  which is {\sl  the central extension of the commutative algebra  $ A_0 $  by the  $ A_0 $--module  $ A_1 \, $}:

\vskip7pt

\begin{proposition} \label{semidirect}
  Let  $ \, G $  be a Chevalley supergroup functor, and let  $ \bG $  be its associated Chevalley supergroup.  Assume  $ \, A \in \salg \, $  is such that  $ \, A_1^2 = \{0\} \, $,  and let  $ \; N_\pm = \big| \Delta_1^\pm \big| \; $.
                                                                    \par
   Then  $ \, G_1^+(A) \, $,  $ G_1^-(A) $  and  $ G_1(A) $  are  {\sl normal subgroups}  of\/  $ G(A) \; $,  \, and we have
 \vskip-15pt
  $$  \displaylines{
   G_1^\pm(A)  \; = \;  G_1^{\pm,<}(A)  \; \cong \;  \mathbb{A}_\bk^{0|N_\pm}\!(A) \;\; ,  \qquad
 G_1(A)  \; = \;  G_1^-(A) \cdot G_1^+(A)  \; = \;   G_1^+(A) \cdot G_1^-(A)  \cr
   G_1(A)  \; \cong \;  G_1^-(A) \times G_1^+(A)  \; \cong \;  G_1^+(A) \times G_1^-(A)  \; \cong \; \mathbb{A}_\bk^{0|N_-}\!(A) \times \mathbb{A}_\bk^{0|N_+}\!(A)  }  $$
%
% (where  ``$ \; \cong $''  means isomorphic as groups), the group
(where  ``$ \; \cong $''  means isomorphic as groups), the group
structure on  $ \mathbb{A}_\bk^{0|N_\pm}\!(A) $  being the obvious
one.  In particular,  $ \; G(A) \, \cong \,  G_0(A) \ltimes G_1(A)  \, \cong \,  G_0(A_0) \ltimes \Big( \mathbb{A}_\bk^{0|N_-}\!(A) \times \mathbb{A}_\bk^{0|N_+}\!(A) \Big) \, $,  the semidirect product of the  {\sl classical}  group  $ G_0(A_0) $  with the  {\sl totally odd}  affine superspace  $ \, \mathbb{A}_\bk^{0|N_-}\!(A) \times \mathbb{A}_\bk^{0|N_+}\!(A) \, $.
 \vskip4pt
   Similar results hold with a symbol  ``$ \, \bG $''  replacing  ``$ \, G $''  everywhere.
\end{proposition}

\medskip

\section{Independence of the construction, Lie's Third Theorem}
\label{indep_Lie-III}

\smallskip

   {\ } \quad   In this section, we discuss the dependence on $ V $  of the Chevalley supergroups  $ \bG_V \, $,  and a super-analogue of Lie's Third Theorem for  $ \bG_V $  and its converse.  Here again, we follow  \cite{fg1},  \S\S 5.4--5.

\medskip

  \subsection{Independence of Chevalley and Kostant superalgebras}  \label{ind_che-kost_alg}

\smallskip

   {\ } \quad   The construction of Chevalley supergroups depend on the finite dimensional, rational  $ \fg $--represen\-tation  $ V $  fixed from scratch, and on an admissible  $ \Z $--lattice  $ M $  in  $ V $.  I show now how it depends on  $ V \, $:  as a consequence, one finds that it is in fact  {\sl independent of  $ M $}.  Once again, I stick to statements, as the proofs follow the same arguments as in  \cite{fg1},  \S\S 5.4--5.

\smallskip

   Let  $ \bG' $  and  $ \bG $  be two Chevalley supergroups obtained by the same  $ \fg \, $,  possibly with a different choice of the representation.  We denote with  $ X_\alpha $  and with  $ X'_\alpha $ respectively the elements of the Chevalley basis in  $ \fg $  identified (as usual) with their images under the two representations of  $ \fg \, $.
                                                                 \par
   Our first result is technical, yet important:

\smallskip

\begin{lemma}
 Let  $ \, \phi : \bG \lra \bG' \, $  be a morphism of Chevalley supergroups such that for all local superalgebras  $ A $  we have
 $ \; \phi_A\big(\bG_0(A)\big) = \bG'_0(A) \; $  and  $ \; \phi_A\big(1 + \vartheta \, X_\beta\big) = 1 + \, \vartheta \, X'_\beta \; $  for all  $ \, \beta \in \Delta_1 \, $,  $ \, \vartheta \in A_1 \, $.  Then  $ \; \text{\it Ker}(\phi_A) \subseteq \bT \; $,  \; where  $ \bT $  is the maximal torus in the group  $ \, \bG_0 \subseteq \bG \, $  (see  \S \ref{const-che-sgroup}).
\end{lemma}

\smallskip

   Let  $ L_V $  be the lattice spanned by the weights in the  $ \fg $--represen\-tation  $ V \, $.  The relation between Chevalley supergroups attached to different weight lattices is the same as in the classical setting:

\medskip

\begin{theorem} \label{mainthm}
 Let  $ \, \bG $  and  $ \, \bG' $  be two Chevalley supergroups constructed using two representations  $ V $  and  $ V' $  of the same  $ \fg $  over the same field\/  $ \KK $  (as in  \S \ref{adm-lat}).  If  $ \, L_V \supseteq L_{V'} \, $,  then there exists a unique morphism  $ \; \phi : \bG \longrightarrow \bG' \; $  such that  $ \; \phi_A\big(1 \! + \vartheta \, X_\alpha\big) = 1 + \vartheta \, X'_\alpha \; $,  \, and  $ \; \text{\it Ker}\,(\phi_A) \subseteq Z\big(\bG(A)\big) \, $,  \, for every local algebra  $ A \, $.  Moreover,  $ \phi $  is an isomorphism if and only if  $ \; L_V = L_{V'} \; $.
\end{theorem}

\smallskip

   As a direct consequence, we have the following ``independence result'':

\smallskip

\begin{corollary}
 Every Chevalley supergroup  $ \, \bG_V $  is independent   --- up to isomorphism ---   of the choice of an admissible lattice  $ M $  of  $ V $  considered in the very construction of  $ \, \bG_V $  itself.
\end{corollary}

\medskip

  \subsection{Lie's Third Theorem}  \label{Lie-III-thm}

\smallskip

   {\ } \quad   In the present context, the analogue of ``Lie's Third Theorem'' concerns the question of whether the tangent Lie superalgebra of our supergroups  $ \bG $  is  $ \, \fg = D(2,1;a) \, $.  We shall now answer it.

\smallskip

   {\sl Let now  $ \,\bk $  be a  {\it field}  (N.B.: this assumption makes the discussion simpler, but it may be dropped, if one acts as in  \cite{ga}, \S 4.6)}.  Let  $ \bG_V $  be a Chevalley supergroup scheme over  $ \bk $,  built out of a  $ \, \fg = D(2,1;a) \, $  and a rational  $ \fg $--module  $ V $  as in  \S \ref{const-che-sgroup}.  In  \S \ref{adm-lat},  we have constructed the Lie superalgebra  $ \, \fg_\bk := \bk \otimes_{\Z_a} \fg_V \, $  over  $ \bk $  starting from the  $ \Z_a $--lattice  $ \fg_V \, $.  We now show that the affine supergroup  $ \bG_V $  has  $ \fg_\bk $  as its tangent Lie superalgebra.

\medskip

   It is well-known that one can associate a Lie superalgebra to a supergroup scheme, in a functorial way.  Let us remind it quickly (and refer to  \cite{ccf}  for further details).

\medskip

   Let  $ \, A \in \salg \, $  and let  $ \, A[\epsilon] := A[x]\big/\big(x^2\big) \, $  be the  {\sl super\/}algebra  of dual numbers, in which  $ \; \epsilon := x \! \mod \! \big(x^2\big) \; $  is taken to be  \textit{even}.  We have that  $ \, A[\epsilon] = A \oplus \epsilon A \, $,  and there are two natural morphisms  $ \; i : A \longrightarrow A[\epsilon] \, , \, a \;{\buildrel i \over \mapsto}\; a \, $,  and  $ \; p : A[\epsilon] \longrightarrow A \, $,  $ \, (a + \epsilon \, a' \big) \;{\buildrel p \over \mapsto}\; a \; $,  \; such that  $ \; p \circ i= {\mathrm{id}}_A \; $.

\medskip

\begin{definition} \label{supergroup}
   For each supergroup scheme  $ G \, $,  consider  $ \; G(p): G (A(\epsilon)) \longrightarrow G(A) \; $. Then there is a supergroup functor  $ \; \Lie(G) : \salg \lra \sets \; $  given on objects by  $ \; \Lie(G)(A) := \text{\it Ker}\,(G(p)) \, $.
\end{definition}

\smallskip

   One shows that the functor  $ \Lie(G) $  is represented by a super vector space, which can be identified with (the functor of points of) the tangent (super)space at the identity of  $ G \, $.  Then by an abuse of notation one denotes by the same symbol both the functor and its representing super vector space. One also proves that the functor  $ \Lie(G) $  takes values in the category  $ \text{\rm (Lie-alg)} $  of  Lie  $ \bk $--algebras  (this is equivalent to saying that the super vector space representing it is a Lie  $ \bk $--superalgebra).
                                                  \par
   Now we compare with  $ \, \fg = D(2,1;a) \, $  the Lie superalgebra  $ \Lie(\bG_V) $  of any of our Chevalley supergroups  $ \bG_V \, $:  the outcome is  (like in  \cite{fg1},  Theorem 5.35, with the same proof):

\smallskip

\begin{theorem}  \label{Lie(G)=g}
 If  $ \, \bG_V $  is a Chevalley supergroup built upon\/  $ \fg $  and  $ V $,  then  $ \; {\Lie}(\bG_V) = \fg \; $  as functors with values in  $ \, \text{\rm (Lie-alg)} \, $.
\end{theorem}

\vskip5pt

\begin{remark}
 \, In the proof of  Theorem \ref{Lie(G)=g}  above, one uses also the fact that the (classical) group  $ \, P_a : \alg \rightarrow \grps \, $  has the functor  $ \, {\Lie}(P_a) : \alg \rightarrow \text{\rm (Lie-alg)} \, , \, A \mapsto A \, $,  as tangent Lie algebra.  The same also holds for the group  $ \, \text{\it Up} : \alg \rightarrow \grps \, $,  which is relevant for  Remark \ref{rem_alt-constr}{\it (b)}.
\end{remark}


\bigskip
\medskip


\begin{thebibliography}{99}

\medskip

\bibitem{bk} J.~Brundan, A.~Kleshchev,  {\it Modular representations of the supergroup  $ Q(n) $,  I},  J.~Algebra  {\bf 206}  (2003), 64--98.

\bibitem{bku} J.~Brundan, J.~Kujava,  {\it A New Proof of the Mullineux Conjecture},  J.~Alg.~Combinatorics  {\bf 18}  (2003), 13--39.

\bibitem{bmpz} Y.~A.~Bahturin, A.~A.~Mikhalev, V.~M.~Petrogradsky, M.~V.~Zaicev,  {\it Infinite-dimensional Lie superalgebras},  de Gruyter Expositions in Mathematics  {\bf 7},  Walter de Gruyter \& Co., Berlin, 1992.

\bibitem{ccf} C.~Carmeli, L.~Caston, R.~Fioresi, {\it  Mathematical Foundations of Supersymmetry},  EMS Series of Lectures in Mathematics  {\bf 15},  European Mathematical Society, 2011.

\bibitem{du} M.~Duflo,  {\it Private communication},  2011.

\bibitem{fg1} R.~Fioresi, F.~Gavarini,  {\it Chevalley Supergroups},  Memoirs AMS  {\bf 215}  (2012), no.~1014.

\bibitem{fg2} R.~Fioresi, F.~Gavarini,  {\it On the construction of Chevalley supergroups},  Supersymmetry in Mathematics \& Physics (UCLA Los Angeles, U.S.A. 2010), 101--123, Lecture Notes in Mathematics  {\bf 2027},  Springer-Verlag, Berlin-Heidelberg, 2011.

\bibitem{fss} L.~Frappat, P.~Sorba, A.~Sciarrino,  {\it Dictionary on Lie algebras and superalgebras},  Academic Press, Inc., San Diego, CA, 2000.

\bibitem{ga} F.~Gavarini,  {\it Algebraic supergroups of Cartan type},  Forum Mathematicum (to appear), 92 pages.

\bibitem{hu} J.~E.~Humphreys,  {\it Introduction to Lie Algebras and Representation Theory},  Graduate Texts in Math., 9, Springer-Verlag, New York, 1972.

\bibitem{ik} K.~Iohara, Y.~Koga,  {\it Central extensions of Lie Superalgebras},
Comment.~Math.~Helv.~{\bf 76}  (2001), 110--154.

\bibitem{ka} V.~G.~Kac,  {\it Lie superalgebras},  Adv.~in Math.~{\bf 26}  (1977), 8--26.

%
% \bibitem{ma} Y.~Manin,  {\it Gauge field theory and complex geometry},  Springer-Verlag, Berlin, 1988.
%

\bibitem{ms} Masuoka, A. {\it The fundamental correspondences in super affine groups and super formal groups},  J.~Pure Appl.~Algebra  {\bf 202}  (2005), 284--312.

\bibitem{sc} M.~Scheunert,  {\it The Theory of Lie Superalgebras},  Lecture Notes in  Math.~{\bf 716}, Springer-Verlag, Berlin-Heidelberg-New York, 1979.

%
% \bibitem{st} R.~Steinberg,  {\it Lectures on Chevalley groups},  Yale University, New Haven, Conn., 1968.
%

\bibitem{sw} B.~Shu, W.~Wang,  {\it Modular representations of the ortho-symplectic supergroups},  Proc.~Lond.{} Math.~Soc.~(3)  {\bf 96}  (2008), 251--271.

%
% \bibitem{vis} E.~G.~Vishnyakova,  {\it On complex Lie supergroups and homogeneous split supermanifolds},
% preprint  {\tt arXiv:0908.1164v1}  [math.DG]  (2009).
%

\bibitem{vsv} V.~S.~Varadarajan,  {\it Supersymmetry for mathematicians: an introduction},  Courant Lecture Notes  {\bf 1},  AMS, 2004.

\end{thebibliography}
\end{document}